\documentclass[aap,preprint]{imsart}

\RequirePackage{amsthm,amsmath,amsfonts,amssymb,dsfont,bbm, subfigure, bm}
\RequirePackage[numbers]{natbib}
\RequirePackage[colorlinks,citecolor=blue,urlcolor=blue]{hyperref}
\RequirePackage{graphicx}

\startlocaldefs
\numberwithin{equation}{section}
\theoremstyle{plain}
\newtheorem{theorem}{Theorem}[section]
\newtheorem{proposition}{Proposition}[section]
\newtheorem{lemma}{Lemma}[section]
\theoremstyle{remark}
\newtheorem{alg}{Algorithm}[section]
\newtheorem{definition}{Definition}[section]
\newtheorem{remark}{Remark}[section]

\newtheorem{assumptions}{Assumptions}[section]
\DeclareMathOperator{\E}{E}
\DeclareMathOperator{\Var}{Var}
\DeclareMathOperator{\Law}{Law}
\DeclareMathOperator{\Prob}{P}

\DeclareMathOperator{\Cov}{Cov}

\newcommand{\Dt}{\Delta t}

\endlocaldefs

\usepackage{enumerate}

\begin{document}

\begin{frontmatter}
\title{A splitting method to reduce MCMC variance}
\runtitle{A splitting method to reduce MCMC variance}

\begin{aug}
\author[A]{\fnms{Robert J.} \snm{Webber}\ead[label=e1]{rw2515@nyu.edu}},
\author[B]{\fnms{David} \snm{Aristoff}\ead[label=e2]{aristoff@math.colostate.edu}}
\and
\author[C]{\fnms{Gideon} \snm{Simpson}\ead[label=e3]{grs53@drexel.edu}}
\address[A]{Courant Institute of Mathematical Sciences, \printead{e1}}

\address[B]{Colorado State University, \printead{e2}}

\address[C]{Drexel University, \printead{e3}}
\end{aug}

\begin{abstract}

We explore whether splitting and killing methods can improve the accuracy of
Markov chain Monte Carlo (MCMC) estimates of rare event probabilities,
and we make three contributions.
First, we prove that ``weighted ensemble" is the only splitting and killing method
that provides asymptotically consistent estimates when combined with MCMC.
Second, we prove
a lower bound on the asymptotic variance 
of weighted ensemble's estimates.
Third, we give a constructive proof and numerical examples to show that
weighted ensemble
can approach this optimal variance bound, in many cases reducing the variance of MCMC estimates by multiple orders of magnitude.

\end{abstract}



\end{frontmatter}


\section{Introduction}

Markov chain Monte Carlo (MCMC) is a stochastic method that empowers researchers to calculate statistics of high-dimensional systems that could not be calculated by other means.
The MCMC approach for calculating an integral $\mu\left(f\right) = \int \mu\left(\mathop{dx}\right) f\left(x\right)$
involves simulating a Markov chain $X_t$ that is ergodic with respect to $\mu$ and then forming the trajectory average
\begin{equation}
\label{eq:trajectory_average}
    \mu\left(f\right)
    \approx \frac{1}{T} \sum_{t=0}^{T-1} f\left(X_t\right).
\end{equation}

Here, we use a broader definition of MCMC than is typical.
We refer to MCMC as any scheme that computes ergodic averages by simulating a Markov chain and then taking trajectory averages.
Our definition thus includes traditional MCMC samplers, such as the random walk Metropolis \cite{metropolis1953equation}
or Gibbs sampler \cite{geman1984stochastic}, that require specifying a target density known up to a normalization constant.
Our definition also extends to samplers where the ergodic distribution is unknown and can only be ascertained through simulations (e.g., \cite{huber1996weighted, costaouec2013analysis}).

Despite the many benefits of MCMC, the approach 
often performs poorly when estimating probabilities of rare sets.
When calculating a small probability $\mu\left(A\right) \ll 1$,
MCMC requires a long simulation time to ensure accuracy, 
and running a simulation for such a long time can be prohibitively computationally expensive.
This limitation makes MCMC difficult to apply in
impactful rare event estimation problems
where accurate computations are greatly needed.

In this work, we investigate the possibility of incorporating \emph{splitting and killing} into MCMC to better compute small probabilities.
Splitting and killing (commonly abbreviated ``splitting") is an approach in which we simulate a collection of Markov chains (``particles") using a common transition kernel $K$.
Periodically, we replicate some of the particles to promote progress toward a rare outcome
and randomly kill other particles to prevent a population explosion.

Splitting is one of
the most prevalent
Monte Carlo approaches for rare event sampling \cite{rubino2009rare}. 
This approach has been developed over seventy years of applications
\cite{rosenbluth1955monte, grassberger1997pruned, glasserman1999multilevel, del2005genealogical, cerou2007adaptive}, 
originally stemming from an idea proposed by John von Neummann in the 1940s \cite{kahn1951estimation}.
Given this long history and the method's demonstrated track record of success
\cite{cerou2019adaptive, hussain2020studying},
we sought to apply splitting to improve MCMC's accuracy for rare event probability estimation.

However, we acknowledge two factors separating MCMC from splitting as traditionally applied.
First, an MCMC method continues for as long as necessary to ensure robust estimates, 
whereas a splitting method typically ends as soon as the particles reach a predetermined stopping time
\cite{liu2008monte, Rubinstein_2016}.
Second, an MCMC method 
uses every data point to compute time averages
\eqref{eq:trajectory_average},
whereas many splitting methods use only the particles' locations at the final algorithmic step \cite{rubino2009rare} to compute estimates.

In this work, we consider a nontraditional approach to splitting that incorporates an arbitrarily long run time and time-averaged estimates.
We ask, what happens if we run an ensemble of ergodic Markov chains and apply splitting at regular intervals?
As time goes on, can splitting improve the accuracy of MCMC estimates?

Through numerical analysis and computational examples, we begin to answer these questions.
We show as $T \rightarrow \infty$ many splitting methods
experience a catastrophic shrinking of statistical weights,
causing all estimates to converge to zero.
Moreover, under mild assumptions, we prove that the \emph{only} splitting method providing consistent estimates as $T \rightarrow \infty$ is the weighted ensemble (WE) method proposed by Huber and Kim in 1997 \cite{huber1996weighted}.

Unique among splitting methods,
WE is characterized by a binning procedure applied at every splitting step.
The particles are divided into bins,
the population in some of the bins is increased through splitting,
and the population in other bins is decreased through killing.
During this splitting step, at least one particle must remain in each bin, as shown below in Figure \ref{fig:WEDiagram} below.

\begin{figure}[h!]
    \centering
    \subfigure[Initial distribution]{\includegraphics[width=6.5cm]{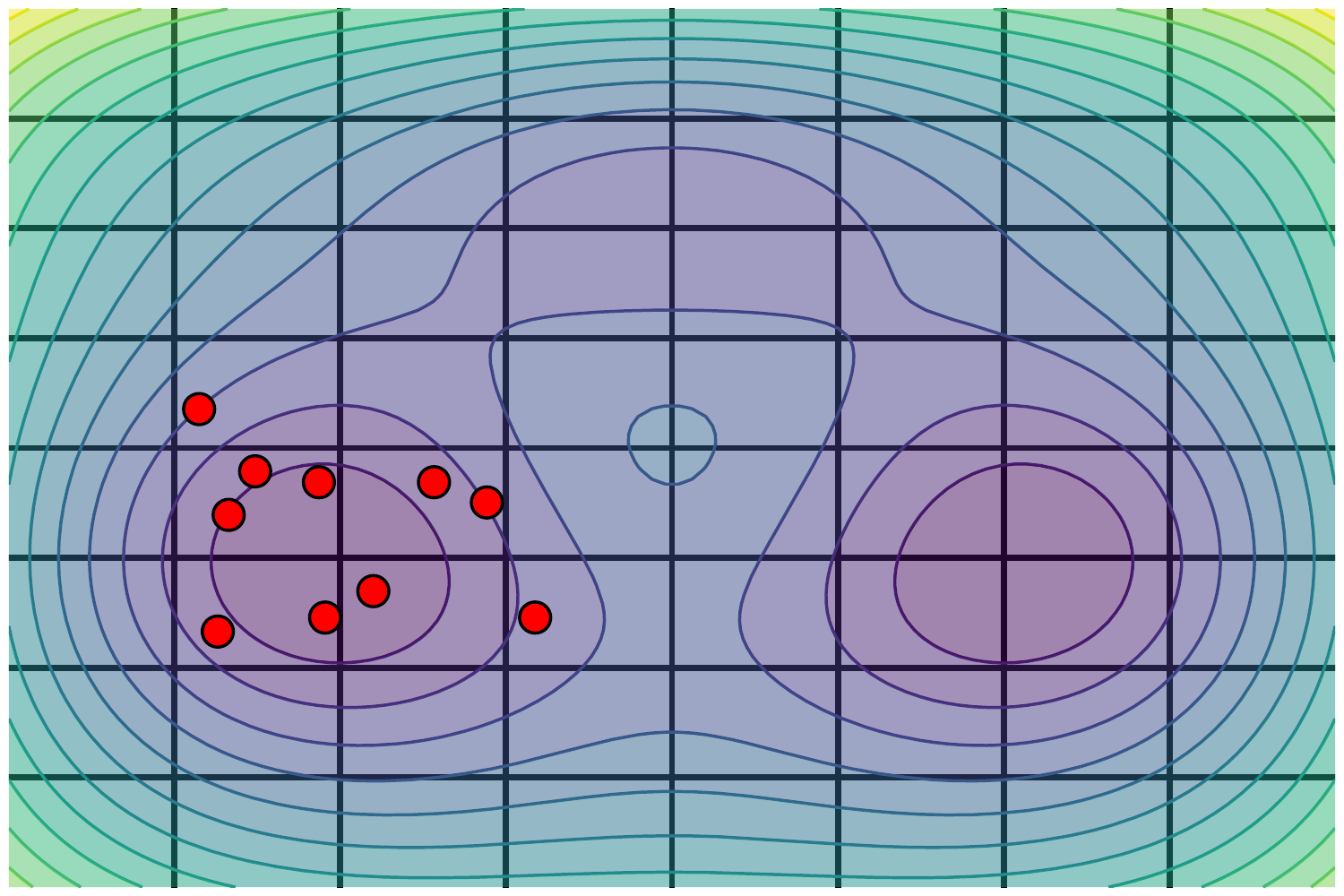}}
    \subfigure[Splitting and killing]{\includegraphics[width=6.5cm]{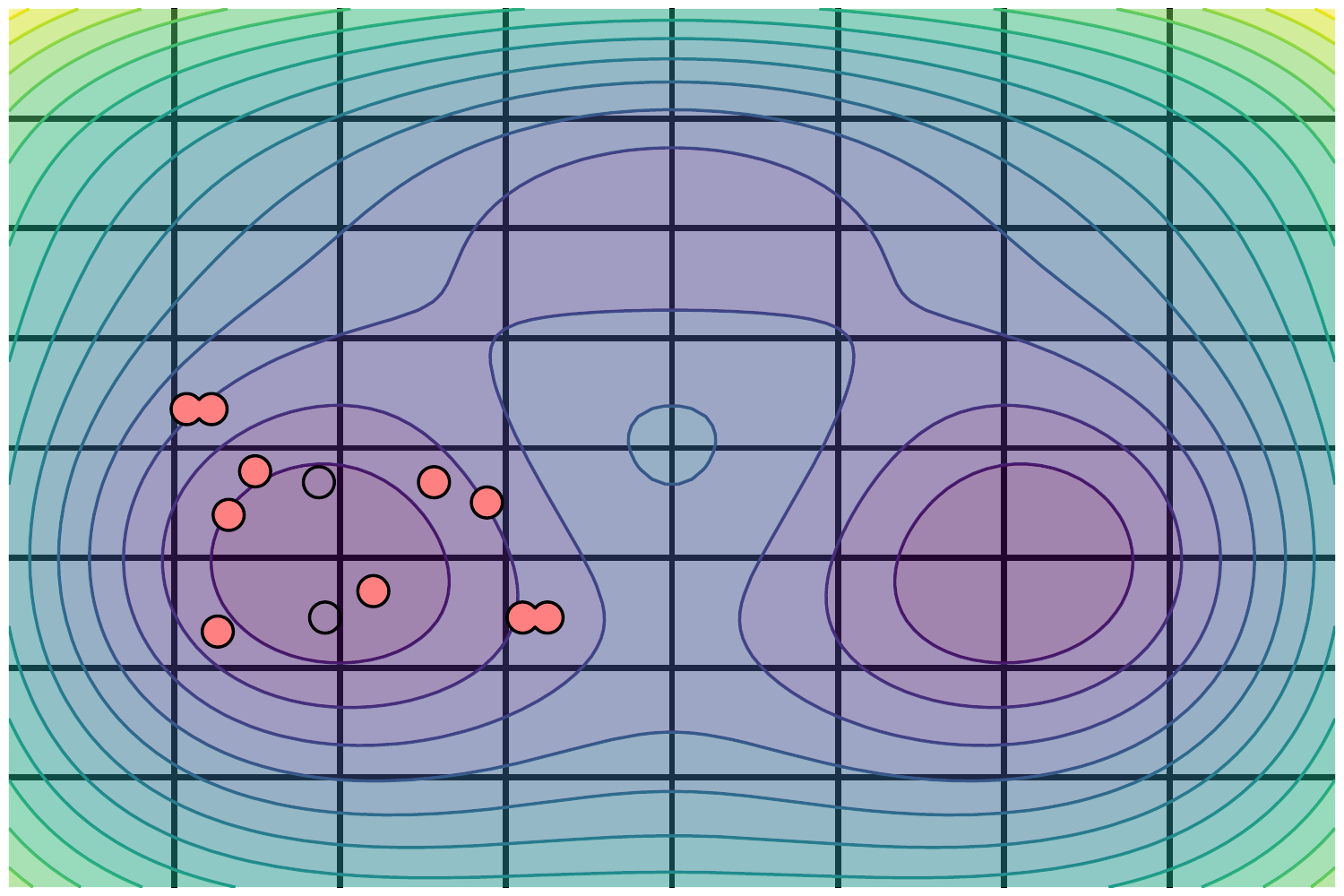}}
    
    \subfigure[Evolution]{\includegraphics[width=6.5cm]{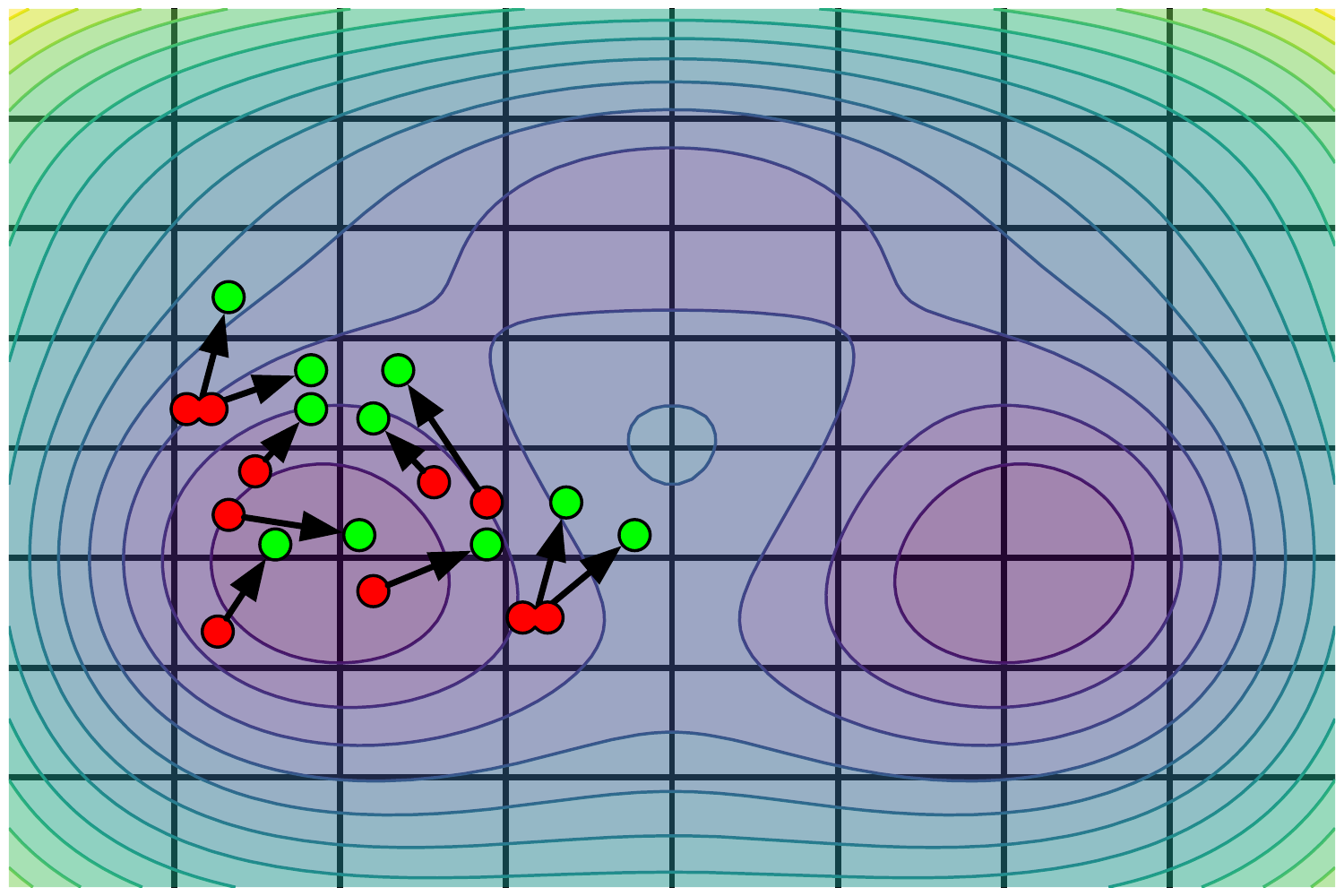}}
    \subfigure[New distribution]{\includegraphics[width=6.5cm]{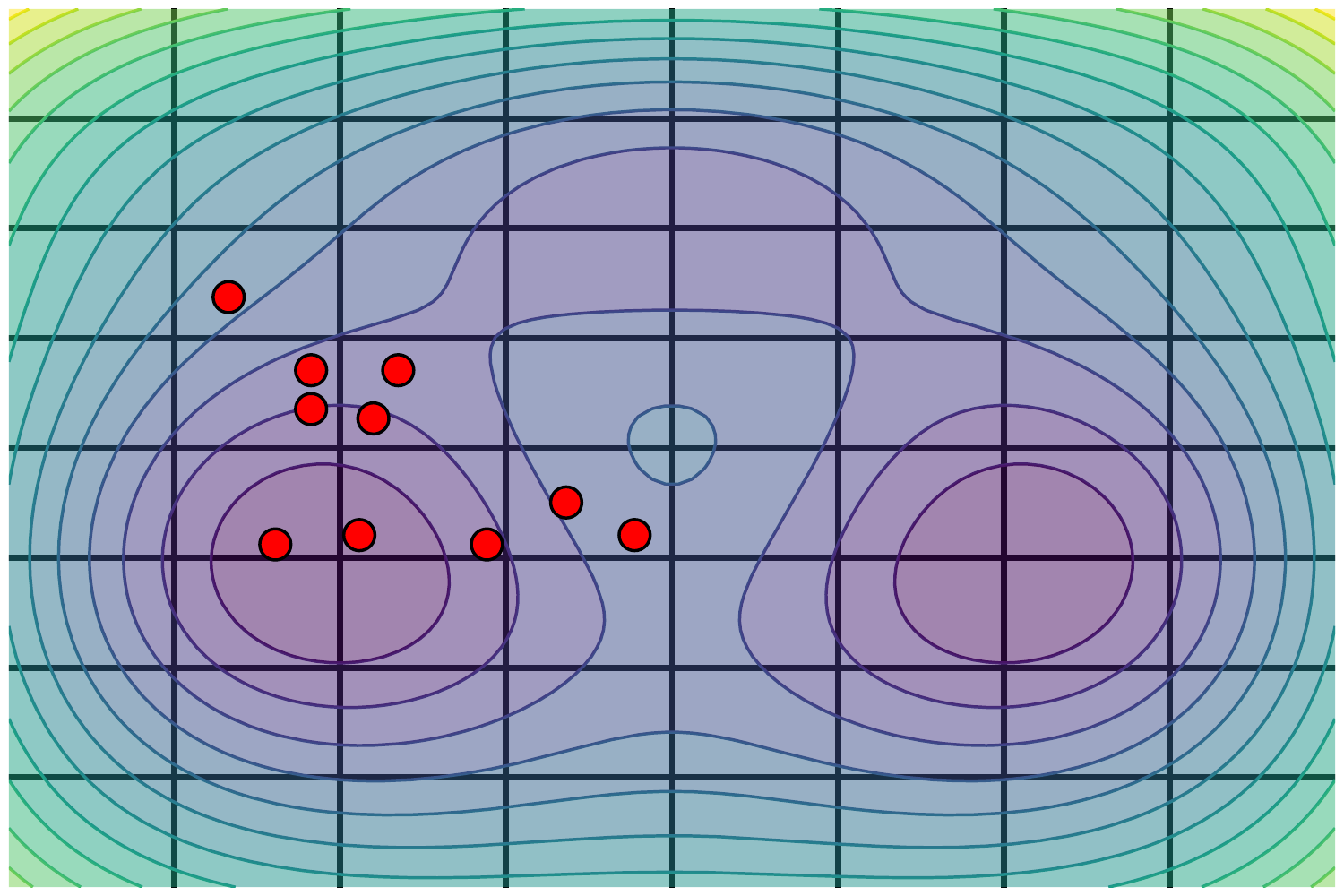}}
    \caption{WE schematic with Cartesian bins and particles diffusing on an energy landscape.
    }
    \label{fig:WEDiagram}
\end{figure}

Researchers have applied WE with increasing frequency over the past decade \cite{bhatt2010steady, zuckerman2017weighted},
and this growing record of applications demonstrates key features that make WE useful:

\begin{enumerate}
    \item WE relies on forward simulations of a process without introducing bias \cite{zhang2010weighted}.
    This makes WE a non-intrusive method suitable for use with any black-box MCMC sampler.
    Thus, WE is used in non-equilibrium statistical mechanics,
    where the invariant distribution is unknown
    and traditional samplers would not be applicable \cite{huber1996weighted,costaouec2013analysis,zuckerman2017weighted}.
    \item WE only requires storage at the precise times of splitting/killing.
    Otherwise, WE avoids recording the complete history of the system, which could be a costly procedure in a high-dimensional state space \cite{dinner2018trajectory}.
\end{enumerate}

As WE has become popular,
there have been efforts toward extracting the method's mathematical properties.
Zhang and coauthors \cite{zhang2010weighted} established the bias properties of WE estimates,
Aristoff \cite{aristoff2019ergodic} established the convergence of WE estimates as $T \rightarrow \infty$,
and Aristoff and Zuckerman \cite{aristoff2018analysis, aristoff2020optimizing}
developed strategies toward algorithmic optimization.
In this past work, however, it was not reported that WE is the \emph{unique}
splitting scheme providing consistent estimates as $T \rightarrow \infty$. 
Additionally, major questions remain open about WE's efficiency
relative to direct MCMC sampling.
In particular, we ask:
when does WE produce more accurate estimates than MCMC?
Also, what is the lowest possible variance that WE estimates can exhibit?

To help answer these questions,
we investigate the asymptotic variance of WE estimates
as $T \rightarrow \infty$.
We establish a lower bound on WE's asymptotic variance that is
valid for any number of particles $N$,
and we prove that WE can come arbitrarily close to achieving this optimal variance bound as $N \rightarrow \infty$.
Additionally, we present examples of rare event estimation problems
where WE reduces MCMC's asymptotic variance by multiple orders of magnitude.
In our numerical experiments, by incorporating sufficiently many particles and optimizing bin allocations,
we obtain nearly the optimal variance reduction that WE offers.

Taken as the whole, the impact of our work is both computational and mathematical.
On the computational side, we demonstrate how an optimized WE method gives accurate
rare event probability estimates that would be extremely costly 
to obtain by direct MCMC sampling.
On the mathematical side, we
show that despite the apparent complexity of a splitting scheme's dynamics,
the stability and variance properties are governed by simple, fundamental considerations in the limit as $T \rightarrow \infty$.

In the rest of this introductory section, we describe our contributions in greater detail
and we lay out the plan for the rest of the paper.

\subsection{Ergodicity theory for splitting schemes}

Our first contribution is to explain how a splitting scheme's
\emph{statistical weights} influence the method's long-time stability.
These statistical weights, which we denote
$w_t^1, \ldots, w_t^{N_t}$,
are assigned to each of the sampled particles $\xi_t^1, \ldots, \xi_t^{N_t}$.
Throughout the scheme, 
the weights are adjusted in inverse proportion to the amount of splitting and killing that occurs.
For example, if a particle is split into two copies, each child receives half the weight of the parent.
Conversely, if two equally weighted particles are randomly reduced to a single particle,
the surviving particle's weight is doubled.
The weights ensure that the splitting scheme can produce estimates
\begin{equation}
\label{eq:splitting_estimates}
     \mu\left(f\right) \approx \frac{1}{T} \sum_{t=0}^{T-1} \sum_{i=1}^{N_t} w_t^i f\left(\xi_t^i\right),
\end{equation}
and these estimates are asymptotically unbiased as $T \rightarrow \infty$.

While the statistical weights ensure that a splitting method's estimates are asymptotically unbiased,
the weights themselves can still degenerate.
In experiments, we find that statistical weights converge to zero in many splitting methods,
which prevents any possibility of consistent estimation.
Moreover, we give a simple explanation for the shrinking weights
by identifying the sum of weights as a nonnegative martingale.
A nonnegative martingale must converge to a positive number or to zero
in the limit 
as $T \rightarrow \infty$ \cite{kallenberg2006foundations}.
Therefore, when the sum of the weights fluctuates infinitely often by a small percentage
--- as occurs in many splitting methods ---
the weights must shrink to zero.

Through martingale arguments, we establish that a splitting method provides asymptotically consistent estimates if and only if the sum of weights is fixed to one at all time steps.
Moreover, under mild conditions, we show that a splitting method maintaining a constant sum of weights must be a WE method.
Thus, we conclude that WE is the unique splitting method that provides asymptotically consistent estimates as $T \rightarrow \infty$.

\subsection{Variance bounds for weighted ensemble}

Our second contribution is to compare the accuracy of MCMC and WE estimates by considering the asymptotic variance as $T \rightarrow \infty$.
For MCMC, there is a Central Limit Theorem that ensures the convergence in distribution
\begin{equation}
    \sqrt{T} \left(\frac{1}{T} \sum_{t=0}^{T-1} f\left(X_t\right) - \mu\left(f\right)\right)
    \stackrel{\mathcal{D}}{\rightarrow} \mathcal{N}\left(0, \mu\left(v_f^2\right)\right),
\end{equation}
whenever the function $f$ is bounded and the MCMC sampler is geometrically ergodic \cite{jones2004markov}.
In this Central Limit Theorem result, 
the asymptotic variance $\mu\left(v_f^2\right)$
provides a quantitative measure of MCMC's accuracy
and determines the simulation time that is needed to obtain accurate results.
Here, the variance function $v_f^2$ is given explicitly by
\begin{equation}
    v_f = \sqrt{K h_f^2 - \left(Kh_f\right)^2},
    \quad h_f = \sum_{t=0}^{\infty} K^t\left(f - \mu\left(f\right)\right)
\end{equation}
where $K$ is the MCMC sampler's transition kernel \cite[ch. 17]{meyn2012markov}.

Our work contributes new asymptotic variance bounds for WE 
that enable a comparison between the WE and MCMC.
For a WE method with $N$ particles, we establish the lower bound
\begin{equation}
\label{eq:asym_lower}
    \liminf_{T \rightarrow \infty} T \Var\left[\frac{1}{T} \sum_{t=0}^{T-1} \sum_{i=1}^N w_t^i f\left(\xi_t^i \right)\right]
    \geq \frac{\mu\left(v_f\right)^2}{N}.
\end{equation}
Additionally, we prove the 
prefactor $\mu\left(v_f\right)^2$ is as sharp as possible.
For any $\epsilon > 0$,
we construct a WE scheme 
whose asymptotic variance satisfies
\begin{equation}
\label{eq:upper}
    \limsup_{T \rightarrow \infty} T 
    \Var\left[\frac{1}{T} \sum_{t=0}^{T-1} \sum_{i=1}^N w_t^i f\left(\xi_t^i \right)\right]
    \leq \left(1 + \epsilon\right) \frac{\mu\left(v_f\right)^2}{N},
\end{equation}
whenever the number of particles $N$ is sufficiently large.
We prove these bounds for general unbounded functions $f$ under suitable integrability conditions.

Our asymptotic variance lower bound \eqref{eq:asym_lower} is a fundamental result that restricts WE's behavior in all parameter regimes.
This result is especially notable
because many previous variance bounds for splitting schemes were derived in the mean field limit \cite{del2012feynman}
in which the number of particles is large and the aggregate behavior becomes highly predictable.
However, in our proof of the lower bound, we avoid reliance on the mean field limit;
rather, we use direct variance manipulations and the assumption of geometric ergodicity
to obtain a result that is valid for any number of particles $N$.
In contrast, our proof that the optimal asymptotic variance can be approached from above does rely on the mean field limit as $N \rightarrow \infty$.

For large $T$ values,
our results make it possible to quantify the maximal variance reduction of WE over MCMC
in terms of an optimal improvement factor (OIF):
\begin{equation}
    \label{e:OIF}
    \text{OIF} \equiv \frac{\mu(v_f^2)}{\mu(v_f)^2}.
\end{equation}
When the OIF is large, as in many rare event probability estimation problems,
our results guarantee the existence of a WE scheme that greatly increases efficiency compared to MCMC.

\subsection{Examples of weighted ensemble's efficiency}

In numerical examples, we demonstrates WE's usefulness for estimating rare event probabilities.
These examples reveal that WE can provide dramatic benefits over MCMC, improving MCMC's variance by many orders of magnitude.
Indeed, Figure \ref{fig:isingintro} reveals a variance reduction of four orders of magnitude
when calculating rare probabilities involving the Ising model (see Section \ref{sec:ising} for details).

\begin{figure}[h!]
    \centering
    \includegraphics[width=8cm]{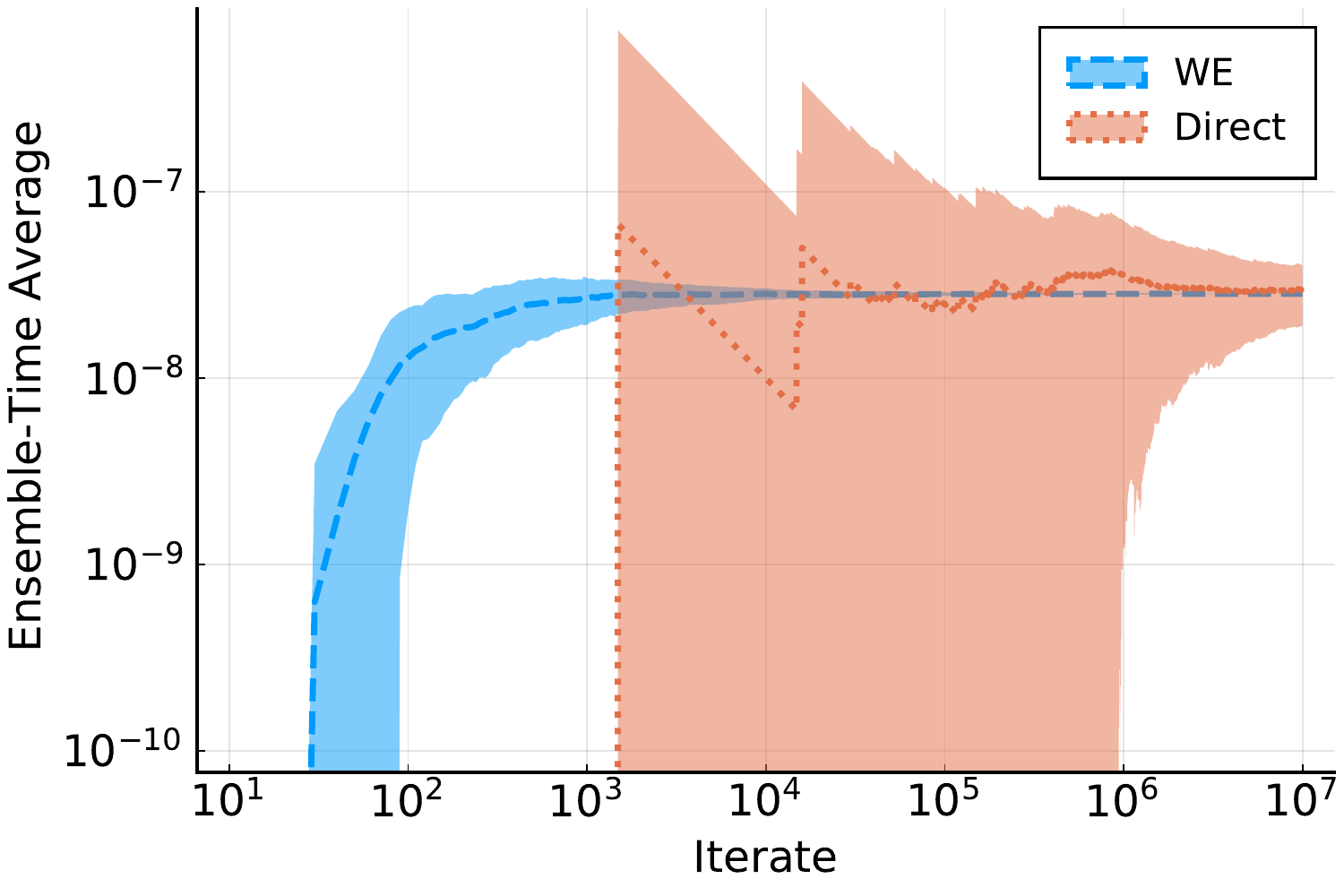}
    \caption{Estimated probability of high magnetization in the high-temperature Ising model,
    as described in Section \ref{sec:ising}.
    The shaded region reflects one sample standard deviation, measured using $100$ independent runs.}
    \label{fig:isingintro}
\end{figure}

These experiments add to the literature that demonstrates major efficiency gains by switching from MCMC to WE (e.g., \cite{pratt2019extensive}).
These experiments also
highlight a surprising consequence of our analysis.
Typically, we would expect a Markov chain-based sampler to be less efficient than a sampler that directly draws independent samples from $\mu$,
because the Markov chain-based samples are positively correlated \cite{sokal1997monte}.
However, in our experiments,
WE produces estimates that are more efficient than could 
possibly be produced by an independence sampler.
In conclusion, our work demonstrates how temporal correlations can be a strength, instead of a weakness,
in rare event sampling.

\subsection{Outline for the paper}

The paper is organized as follows.
Ergodicity theory is presented in Section \ref{sec:special}, 
variance bounds are in Section \ref{sec:variance},
mathematical proofs are in Section \ref{sec:theory},
numerical experiments are in Section \ref{sec:experiments},
and the conclusions follow in Section \ref{sec:conclusion}.

\section{Ergodicity theory for splitting schemes}{\label{sec:special}}

In this section, we define a splitting method and a weighted ensemble (WE) method.
Then, we present our results proving that WE is the only splitting method that provides asymptotically consistent estimates as $T \rightarrow \infty$.
Throughout the analysis, we use $\left\lVert f \right\rVert = \sup_x \left|f\left(x\right)\right|$ to denote the supremum norm on functions
and $\left\lVert \mu \right\rVert = \sup_{\left\lVert f \right\rVert\leq 1} \left|\mu\left(f\right)\right|$ to denote the total variation norm on measures.
We defer the technical proofs to Section \ref{sec:theory}.

\subsection{Definitions of splitting and weighted ensemble}

A splitting method \cite{liu2008monte, Rubinstein_2016, rubino2009rare} is a Monte Carlo method that alternates between a splitting step and an evolution step as follows.

\begin{alg}[Splitting method]
\label{alg:splitting}
~
\newline
First, independently sample particles
$\xi_0^1, \ldots, \xi_0^{N_0}$ from a distribution $\mu_0$
and set $w_0^i = 1 \slash N_0$ for $i = 1, \ldots, {N_0}$.
Then, apply a splitting step and an evolution step at each time $t = 0, 1, \ldots$
\begin{enumerate}
    \item Given particles and weights $\left(\xi_t^i, w_t^i\right)_{1 \leq i \leq N_t}$,
    apply the following splitting step.
    \begin{enumerate}[a.]
        \item Select the mean number of children $C_t^i > 0$ for each particle $\xi_t^i$.
        \item Select the actual number of children $N_t^i \geq 0$ for each particle $\xi_t^i$, making sure that $N_t^i$ is a nonnegative integer with mean
        $C_t^i$.
        \item Split each particle $\xi_t^i$ into $N_t^i$ copies. 
        \item Assign the children of $\xi_t^i$ uniform weights $w_t^i \slash C_t^i$.
    \end{enumerate}
    \item Given particles and weights $\left(\hat{\xi}_t^i, \hat{w}_t^i\right)_{1 \leq i \leq N_{t+1}}$, apply the following evolution step.
        \begin{enumerate}[a.]
        \item Evolve each particle $\hat{\xi}_t^i$ to a new state $\xi_{t+1}^i$ according to the transition kernel $K$. 
        \item Assign each particle $\xi_{t+1}^i$ a weight $w_{t+1}^i = \hat{w}_t^i$.
    \end{enumerate}
\end{enumerate}
\end{alg}

A splitting method is highly general,
since there are many possible strategies for choosing the numbers $C_t^i$ and $N_t^i$ during the splitting step.
However, the main rule
specified in Algorithm \ref{alg:splitting} is that children of
$\xi_t^i$ receive uniform weights $w_t^i \slash C_t^i$.
This rule ensures that the weights of the children of a $\xi_t^i$ sum up to the weight $w_t^i$ in expectation.
To our knowledge, all 
the most popular splitting schemes are consistent with this rule, 
given appropriate definitions of $N_t^i$ and $C_t^i$.
For example, 
in the original WE method of Huber and Kim \cite{huber1996weighted}, 
the $C_t^i$ are themselves random.
Thus, a particle with weight $w_t^i = 1$ might randomly produce $C_t^i = N_t^i = 2$ copies with weights $1 \slash 2$ or $C_t^i = N_t^i = 3$ copies with weights $1 \slash 3$.

For illustration, we show a typical splitting method in Figure \ref{fig:figure1} below.

\begin{figure}[h!]
	\includegraphics[scale = .4, clip, trim = {0 .5cm 0 .5cm}]{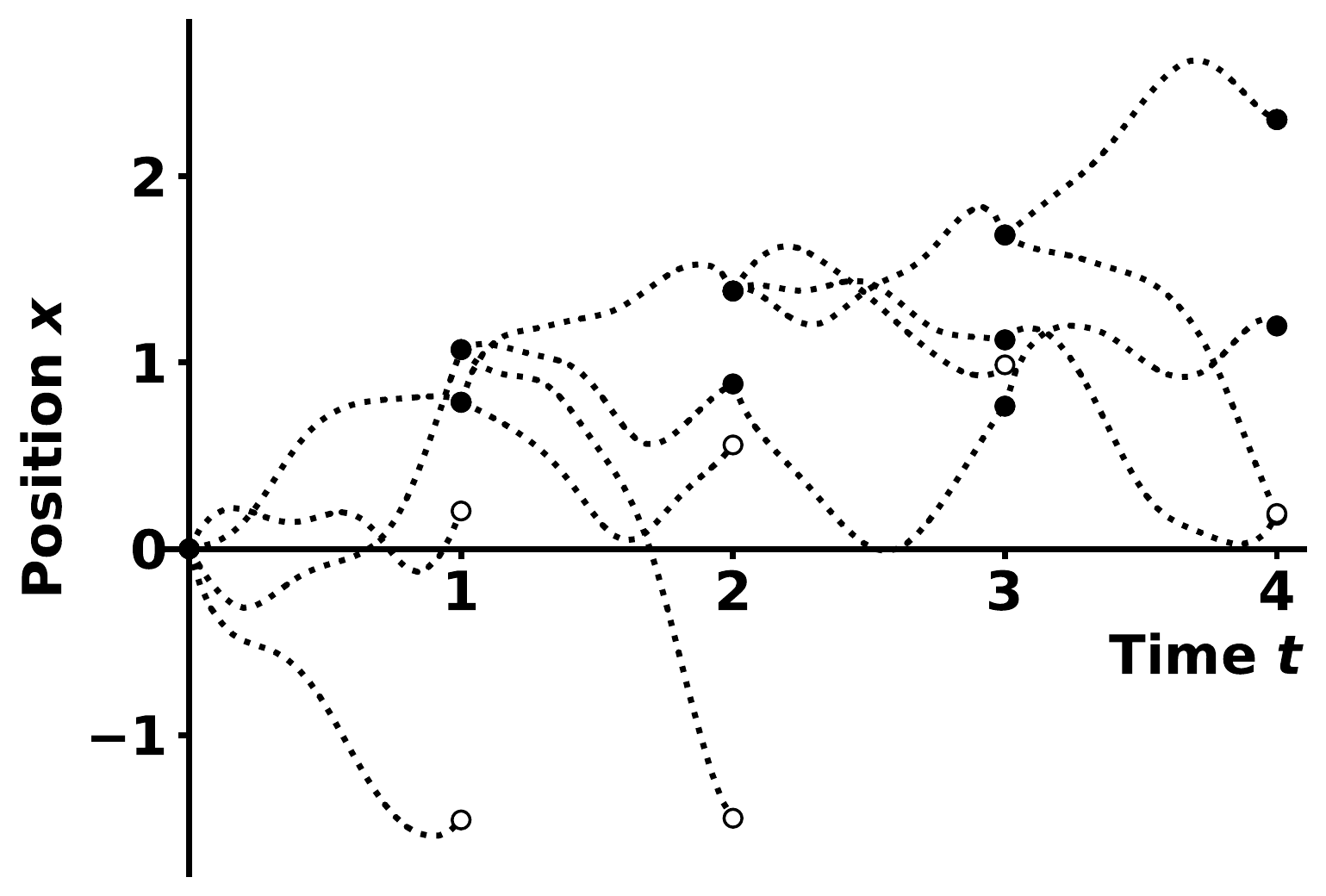}
	\caption{\label{fig:figure1} Splitting is used to sample
	rare, high values of the position $x$.
	White circles indicate that samples are killed.
	Black circles indicate that samples are preserved and possibly copied.}
\end{figure}

A WE method \cite{huber1996weighted, Darve_2012} is a particular type of splitting method
that imposes more structure during the splitting step.
In WE, we first divide particles into bins 
(more precisely, we divide particle indices $1, 2, \ldots, N_t$ into bins).
Then, we use splitting and killing to adjust the populations in the bins
while exactly preserving the bins' statistical weights.

\begin{alg}[Weighted ensemble]
\label{alg:we}
~
\newline
Apply a splitting method, and at each time $t \geq 0$ perform the following splitting step:
\begin{enumerate}[a.]
    \item Partition the indices $1 \leq i \leq N_t$ into bins $u_1, u_2, \ldots$,
    and set $w_t\left(u\right) = \sum_{i \in u} w_t^i$.
    \item Select the desired number of children $N_t\left(u\right) \geq 1$ for each bin $u$.
    \item Select the number of children $N_t^i \geq 0$ for each particle $\xi_t^i$ in each bin $u$, making sure $N_t^i$ is a nonnegative integer with expectation
    $C_t^i = N_t\left(u\right) w_t^i \slash w_t\left(u\right)$ and $\sum_{i \in u} N_t^i = N_t\left(u\right)$.
    \item Split each particle $\xi_t^i$ into $N_t^i$ copies.
    \item Assign the children with parents in bin $u$ 
    uniform weights $w_t\left(u\right) \slash N_t\left(u\right)$.
\end{enumerate}
\end{alg}

Our definition of weighted ensemble 
allows for an arbitrary choice of bins
that may change at each time step $t \geq 0$.
However, in our analysis,
we focus on 
the simple case where bins correspond to fixed regions that divide up the state space,
as was shown in Figure \ref{fig:WEDiagram}.
Past work exploring optimal bin design strategies includes \cite{aristoff2020optimizing,copperman2020accelerated2, copperman2020accelerated,torrillo2020minimal}.

To complete our introduction to splitting and WE methods,
we describe common approaches for selecting a population of children particles $\hat{\xi}_t^1, \ldots, \hat{\xi}_t^{N_{t+1}}$
from a population of parent particles
$\xi_t^1, \ldots, \xi_t^{N_t}$.
We assume that the WE user has already determined that 
each particle $\xi_t^i$ should produce $C_t^i > 0$ children particles on average, 
and we describe several \emph{resampling schemes} \cite{webber2019unifying}
for determining the precise number of children $N_t^i$ for each particle $\xi_t^i$.

The simplest resampling scheme is \emph{multinomial resampling},
which we describe below.

\begin{definition}{\label{def:multi}}
In multinomial resampling \cite{del2012feynman}, 
we independently sample children particles $\hat{\xi}_t^1, \ldots, \hat{\xi}_t^{N_{t+1}}$ from locations $\left(\xi_t^i\right)_{1 \leq i \leq N_t}$
with probabilities proportional to $\left(C_t^i\right)_{1 \leq i \leq N_t}$.
Thus, the numbers $\left(N_t^i\right)_{1 \leq i \leq N_t}$ are jointly distributed according to
\begin{equation}
\label{e:multinomial}
    \left(N_t^1, \ldots, N_t^{N_t}\right)
    \sim \text{Multi}\left(N_t, \frac{C_t^1}{N_t}, \ldots, \frac{C_t^{N_t}}{N_t}\right).
\end{equation}
\end{definition}

Multinomial resampling is used in many splitting schemes \cite{del2005genealogical},
but it cannot be used in weighted ensemble
since it would violate the requirement of placing exactly $N_t\left(u\right)$ children particles in each bin $u$.
However, a related approach called \emph{binned multinomial resampling} can be used with WE instead.

\begin{definition}{\label{def:binned}}
In binned multinomial resampling,
we iterate over the bins and apply multinomial resampling within each bin.
Thus, if the particles in bin $u$
are $\xi_t^{i_1}, \ldots, \xi_t^{i_m}$,
the numbers $N_t^{i_1}, \ldots, N_t^{i_m}$ are jointly distributed according to
\begin{equation}
\label{e:binmultinomial}
    \left(N_t^{i_1}, \ldots, N_t^{i_m}\right) \sim \text{Multi}\left(N_t\left(u\right), \frac{w_t^{i_1}}{w_t\left(u\right)}, \ldots, \frac{w_t^{i_m}}{w_t\left(u\right)}\right).
\end{equation}
\end{definition}

In multinomial resampling and binned multinomial resampling,
we observe that children particles $\hat{\xi}_t^1, \ldots, \hat{\xi}_t^{N_{t+1}}$ 
are independently sampled
given auxiliary information including the locations of the parents
and the mean number of children for each parent.
In general, we can consider other resampling possible schemes that maintain this conditional independence property.
We define the class of \emph{conditionally independent resampling} schemes as follows.

\begin{definition}{\label{def:independent}}
In a conditionally independent resampling scheme \cite{webber2019unifying},
given parent particles $\xi_t^1, \ldots, \xi_t^{N_t}$ with weights
$w_t^1, \ldots, w_t^{N_t}$,
we define a matrix $\bm{P} \in \mathbb{R}^{M \times N_t}$
where $M$ represents the maximum possible number of allowable children.
Then, we iterate over $i = 1, 2, \ldots, M$.
With probability $\bm{P}_{ij}$, we assign
\begin{equation}
    \hat{\xi}_t^i = \xi_t^j
    \quad \text{and} \quad \hat{w}_t^i = \frac{w_t^j}{C_t^j}.
\end{equation}
With the remaining probability $1 - \sum_{j=1}^{N_t} \bm{P}_{ij}$,
we do not assign $\hat{\xi}_t^i$ to any location at all and we set $\hat{w}_t^i = 0$.
We remove particles with zero weights at the end of the resampling scheme.
\end{definition}

Conditionally independent resampling schemes are a broad category that encompasses most procedures
that are used in practice.
In addition to multinomial resampling and binned multinomial resampling,
this category includes Bernoulli resampling, multinomial residual resampling, stratified resampling, and stratified residual resampling \cite{douc2005comparison, webber2019unifying}.

While the choice of resampling scheme can affect a splitting method's variance,
here we do not provide a detailed comparison between different procedures.
Rather, we emphasize broad results that hold for many different resampling schemes.
In our analysis, we only make specific assumptions about the resampling scheme
twice.
First, when we establish that WE is the unique splitting method providing 
asymptotically convergent estimates,
we assume a conditionally independent resampling scheme (see Proposition \ref{prop:uniqueness}).
Second, in our our demonstration that WE can approach the optimal asymptotic variance from above, 
our construction is based on binned multinomial resampling (see Lemma \ref{lem:exact}).

\subsection{Ergodicity of splitting schemes}

When we combine splitting with MCMC,
we must carefully consider the long-time
behavior of the splitting method's estimates.
It is well-known that MCMC estimates must converge
\begin{equation}
\label{eq:asym_consist}
    \frac{1}{T} \sum_{t=0}^{T-1} f\left(X_t\right) \stackrel{T \rightarrow \infty}{\rightarrow} \mu\left(f\right),
\end{equation}
assuming $f$ is bounded and the dynamics are Harris ergodic \cite[ch. 17]{meyn2012markov}.
Therefore, we ask whether splitting estimates also converge similarly to MCMC estimates.
Specifically, we ask: do the estimates from a splitting method always satisfy
\begin{equation}
\label{eq:asym_consist2}
    \frac{1}{T} \sum_{t=0}^{T-1} \sum_{i=1}^N w_t^i f\left(\xi_t^i\right) \stackrel{T \rightarrow \infty}{\rightarrow} \mu\left(f\right),
\end{equation}
and, if not, what assumptions guarantee the convergence in \eqref{eq:asym_consist2}?

To address these questions in a rigorous way,
we first define the ergodicity conditions
that will be considered in the analysis.
\begin{definition}[Ergodicity conditions]
\label{def:ergodicity}
~
\newline
Consider a $\psi$-irreducible, aperiodic transition kernel $K$
on a general state space $X$,
and assume $K$ is invariant with respect to a distribution $\mu = \mu K$.
\begin{enumerate}
\item[(i)]
The kernel $K$ is \emph{Harris ergodic} if
\begin{equation} 
\left\lVert K^t\left(x, \cdot\right) - \mu \right\rVert \stackrel{t \rightarrow \infty}{\rightarrow} 0
\end{equation}
for all $x \in X$.
\item[(ii)]
The kernel $K$ is \emph{geometrically ergodic} if
\begin{equation}
\sum_{t=0}^{\infty} r^t \left\lVert K^t\left(x, \cdot\right) - \mu \right\rVert < \infty
\end{equation}
for fixed $r > 1$ and all $x \in X$.
\item[(iii)] The kernel $K$ is \emph{$V$-uniformly ergodic} for a function $1 \leq V \leq \infty$ if $\mu\left(V\right) < \infty$ and
\begin{equation}
\label{eq:GEOMERG}
\sup_{\left|g\right| \leq V} 
\left|K^t g\left(x\right) - \mu\left(g\right)\right| 
\leq R \rho^t  V(x)
\end{equation}
for fixed $R > 0$, fixed $\rho < 1$, and all $x \in X$.
\end{enumerate}
\end{definition}

Harris ergodicity is a comparatively weak condition
that gives no control over the convergence rate in $\left\lVert K^t\left(x, \cdot\right) - \mu\right\rVert \stackrel{t \rightarrow \infty}{\rightarrow} 0$,
whereas
geometric ergodicity and $V$-uniform ergodicity are stronger conditions that specify an exponential convergence rate.
While geometric ergodicity and $V$-uniform ergodicity are nearly equivalent,
we exploit the slight difference between these conditions in our analysis.
As explained in \cite[ch.15]{meyn2012markov},
geometric ergodicity implies $V$-uniform ergodicity
for a particular function $V$.
Conversely, $V$-uniform ergodicity implies geometric ergodicity if we restrict the process to the absorbing set $\left\{V < \infty\right\}$.

In contrast to an MCMC method, our experiments reveal that a splitting method does not necessarily provide consistent estimates as $T \rightarrow \infty$, even when $K$ is geometrically ergodic and $f$ is bounded.
Figure \ref{fig:converge} shows an example of a splitting scheme that fails to provide consistent estimates.
The sum of the weights approaches zero over long timescales, which causes estimates to converge to zero also.

\begin{figure}[h!]
	\includegraphics[scale = .4]{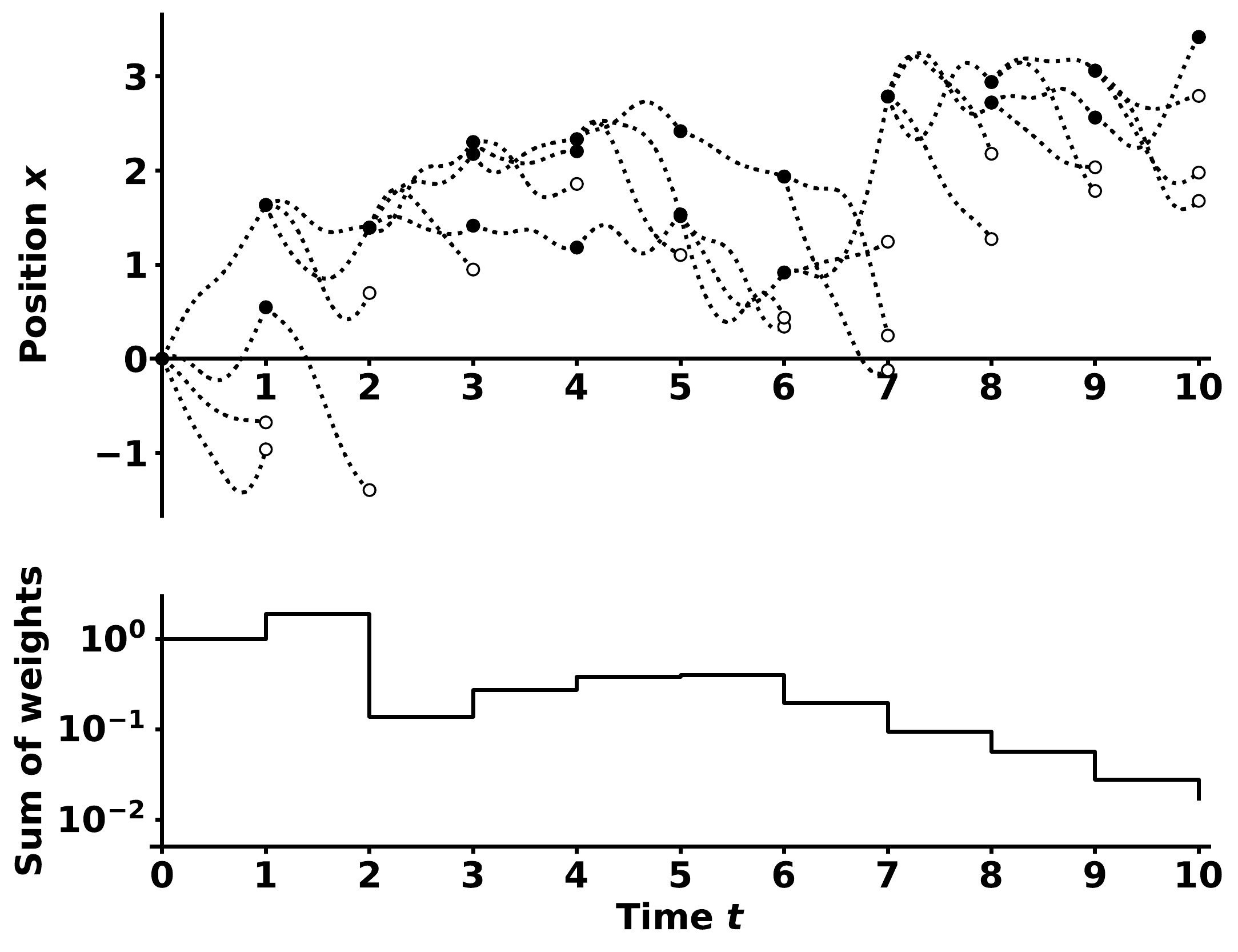}
	\caption{Weights converge to zero as splitting and killing are repeatedly applied.}
	\label{fig:converge}
\end{figure}

This problem of shrinking weights was first pointed out by Aristoff in \cite{aristoff2019ergodic}, 
and here we provide a full mathematical explanation.
Since the sum of the weights is a martingale, small relative fluctuations in the sum of the weights build up over time.
Moreover, the small fluctuations lead to major consequences, as we demonstrate in the proposition below.

\begin{proposition}{\label{prop:collapse}}
In a splitting method, suppose there exists $\epsilon > 0$ such that the event
\begin{equation}
    \left|\frac{\sum_{i=1}^{N_{t+1}} w_{t+1}^i}
    {\sum_{i=1}^{N_t} w_t^i} - 1\right| > \epsilon
\end{equation}
occurs infinitely often with probability one.
Then, almost surely, $\sum_{i=1}^{N_t} w_t^i \rightarrow 0$ as $t \rightarrow \infty$.
\end{proposition}

As a consequence of Proposition \ref{prop:collapse},
the only way to avoid shrinking weights is to asymptotically eliminate all small fluctuations as $T \rightarrow \infty$.
This pressing need to control the sum of the weights leads us to consider the possibility of simply fixing the sum of weights to be one,
as naturally occurs in WE.
In the next theorem, we verify that a splitting scheme
provides asymptotically consistent estimates
if and only if $\sum_{i=1}^{N_t} w_t^i = 1$ at all times $t \geq 0$.

\begin{theorem}{\label{thm:requirement}}
Consider a splitting method with a $V$-uniformly ergodic kernel $K$
and assume $\mu_0\left\{V < \infty\right\} = 1$.
Then, the following three conditions are equivalent:
\begin{enumerate}
\item[(i)] The time average of the sum of the weights converges in probability to one:
\begin{equation}
    \Prob\left\{\left|\frac{1}{T} \sum_{t=0}^{T-1} \sum_{i=1}^{N_t} w_t^i - 1\right| > \epsilon \right\} \stackrel{T \rightarrow \infty}{\rightarrow} 0, \quad \epsilon > 0.
\end{equation}
\item[(ii)] With probability one, the weights satisfy $\sum_{i=1}^N w_t^i = 1$ at all times $t \geq 0$.
\item[(iii)] Whenever $\left\lVert f^2 \slash V \right\rVert < \infty$, the estimates of $\mu\left(f\right)$ converge with probability one:
\begin{equation}\label{eq:ergodictheorem}
    \Prob\left\{\frac{1}{T} \sum_{t=0}^{T-1} \sum_{i=1}^{N_t} w_t^i f\left(\xi_t^i\right) \stackrel{T \rightarrow \infty}{\rightarrow} \mu\left(f\right)\right\} = 1.
\end{equation}

\end{enumerate}
\end{theorem}

Next, we prove that the \emph{only} splitting method capable of maintaining a constant sum of weights is WE,
assuming a conditionally independent resampling scheme is used.

\begin{proposition}{\label{prop:uniqueness}}
If a splitting method 
with a conditionally independent resampling scheme satisfies
$\sum_{i=1}^{N_t} w_t^i = 1$ at all times $t \geq 0$,
the splitting method is a weighted ensemble method
with a particular choice of bins.
\end{proposition}

In summary, we have proved under mild conditions that WE is the only splitting method that provides asymptotically consistent estimates as $T \rightarrow \infty$.

\begin{remark}
In this section, we have analyzed a splitting method's estimates
\begin{equation}
\label{eq:worse}
    \mu\left(f\right) \approx \frac{1}{T} \sum_{t=0}^{T-1} \sum_{i=1}^{N_t} w_t^i f\left(\xi_t^i\right).
\end{equation}
Yet, we might also consider a splitting method's normalized estimates
\begin{equation}
\label{eq:better}
    \mu\left(f\right) \approx \frac{1}{T} \sum_{t=0}^{T-1} \frac{\sum_{i=1}^{N_t} w_t^i f\left(\xi_t^i\right)}
    {\sum_{i=1}^{N_t} w_t^i}.
\end{equation}
In WE, the normalized and unnormalized estimates are the same.
However, in splitting methods other than WE, the normalized and unnormalized estimates differ.
Past analyses of splitting methods \cite{del2012feynman, gubernatis2016quantum} showed that the normalized estimates typically
converge as $T \rightarrow \infty$ with an asymptotic bias of size $\mathcal{O}\left(\frac{1}{N}\right)$,
 preventing the possibility of consistent estimation.
WE presents the only known exception to this trend, since the asymptotic bias is zero.
\end{remark}

\section{The bias and variance of weighted ensemble estimates}{\label{sec:variance}}

In this section, our broad goal is to determine whether WE can
produce more accurate estimates than MCMC.
We first analyze the bias and then analyze the variance of WE estimates.
Throughout the section, we fix the number of particles to be exactly $N$ at all time steps, and we analyze $N$ as a key control parameter influencing WE's efficiency.

\subsection{Bias of weighted ensemble estimates}

As our first result,
we find that WE adds no additional bias compared to MCMC.
Rather, as was originally shown in \cite{zhang2010weighted}, 
WE estimates have the same expectation
as estimates from a standard MCMC sampler.

\begin{proposition}{\label{prop:unbiased}}
Consider a WE scheme with a Harris ergodic kernel $K$,
and assume $f$ is bounded.
WE estimates for $\mu\left(f\right)$ have the following bias properties:
\begin{enumerate}
	\item For any $T \geq 0$, the WE estimate
	\begin{equation}
	    \frac{1}{T} \sum_{t=0}^{T-1} \sum_{i=1}^{N} w_t^i f\left(\xi_t^i\right)
	\end{equation}
    has the same expectation as a trajectory average 
    \begin{equation}
        \frac{1}{T} \sum_{t=0}^{T-1} f\left(X_t\right),
    \end{equation}
	where $X_t$ is a Markov chain with transition kernel $K$
	and initial distribution $\mu_0$.
	\item The WE estimates for $\mu\left(f\right)$ are asymptotically unbiased in the limit as $T \rightarrow \infty$:
	\begin{equation}\label{eq_asymp_unbiased}
	    \E\left[\frac{1}{T} \sum_{t=0}^{T-1} \sum_{i=1}^{N} w_t^i f\left(\xi_t^i\right)\right]
	    \stackrel{T \rightarrow \infty}{\rightarrow} 
	    \mu\left(f\right).
	\end{equation}
	\item 
	With a burn-in period of length $\tau$, the WE estimates for $\mu\left(f\right)$ satisfy
    \begin{equation}
    \label{eq_burn_in}
    \E\left[\frac{1}{T} \sum_{t=\tau}^{\tau + T-1} \sum_{i=1}^{N} w_t^i f\left(\xi_t^i\right)\right]
    \stackrel{\tau \rightarrow \infty}{\rightarrow} 
	\mu\left(f\right).
    \end{equation}
\end{enumerate}
\end{proposition}

In the limit as $T \rightarrow \infty$, Proposition \ref{prop:unbiased} shows that WE is asymptotically unbiased.
However, we may worry about bias in the pre-asymptotic regime.
To reduce bias, therefore, we can run the WE algorithm for an extra $\tau$ time steps
and use the estimate
\begin{equation}
    \mu\left(f\right) \approx \frac{1}{T} \sum_{t=\tau}^{\tau + T-1} \sum_{i=1}^{N} w_t^i f\left(\xi_t^i\right).
\end{equation}
Proposition \ref{prop:unbiased}
verifies that incorporating a ``burn-in period" of length $\tau$ 
has a beneficial impact.
As $\tau \rightarrow \infty$, the bias in WE estimates vanishes completely.

\subsection{Variance of weighted ensemble estimates}{\label{sub:variance}}

Now that we have considered bias,
our next step is evaluating the variance of WE estimates.
To provide simple formulas for the WE variance,
we fix a function $f$ and define the associated conditional expectation function
\begin{equation}
\label{e:hf}
    h_f\left(x\right) = \sum_{t=0}^{\infty} \left(K^t f\left(x\right) - \mu\left(f\right)\right)
\end{equation}
and the function 
\begin{equation}
\label{e:vf}
v_f(x) = \sqrt{K h_f^2(x) - \left(K h_f(x)\right)^2}.
\end{equation}
The function $v_f^2$
is commonly used in the Markov chain literature 
to express the asymptotic variance of trajectory averages involving $f$.
Here, we build on this literature by using $v_f$ to also compare MCMC and WE variances.
Our main result is the following theorem, which establishes the best possible asymptotic variance for WE estimates.

\begin{theorem}\label{thm:bounds}
Consider a WE scheme with a kernel $K$ that is geometrically ergodic and $V$-uniformly ergodic, and assume $\left\lVert f^2 \slash V \right\rVert < \infty$.
WE estimates for $\mu\left(f\right)$ have the following variance properties:
\begin{enumerate}
	\item The variance of WE estimates is bounded from below by
	\begin{equation}\label{eq:WE_lower_bd}
	\liminf_{T \rightarrow \infty} T \Var\left[\frac{1}{T} \sum_{t=0}^{T-1} \sum_{i=1}^N w_t^i f\left(\xi_t^i\right)\right] \geq \frac{\mu\left(v_f\right)^2}{N}.
	\end{equation}
	\item For any $\epsilon > 0$,
	there is a particular WE scheme 
	requiring a sufficiently large number of particles $N$
	that satisfies
	\begin{equation}
	\label{eq:WE_upper_bd}
	\limsup_{T \rightarrow \infty} T \Var\left[\frac{1}{T} \sum_{t=0}^{T-1} \sum_{i=1}^N w_t^i f\left(\xi_t^i\right)\right] \leq \left(1 + \epsilon\right) \frac{\mu\left(v_f\right)^2}{N}.
	\end{equation}
\end{enumerate}
\end{theorem}

Theorem \ref{thm:bounds} opens up the possibility for a quantitative comparison between MCMC's and WE's efficiencies.
For MCMC, we can either run a single Markov chain for $NT$ time steps,
or we can run $N$ independent Markov chains for $T$ time. 
Both approaches share a similar computational cost,
and both approaches yield an estimate of $\mu\left(f\right)$ whose variance is nearly $\mu\left(v_f^2\right) \slash NT$.
In contrast, by running WE for $T$ time steps with $N$ particles,
it may be possible to achieve a much lower variance.
With the optimal design parameters
and with a sufficiently large number of particles $N$,
WE produces an estimate of $\mu\left(f\right)$ whose variance is nearly
$\mu(v_f)^2 \slash NT$.

We can quantify the efficiency benefits of WE over MCMC by means of the optimal improvement factor (OIF) that was previously discussed in the introduction section:
\begin{equation}
    \text{OIF} \equiv \frac{\mu(v_f^2)}{\mu(v_f)^2}.
\end{equation}
If this factor is one, then WE cannot improve MCMC's variance at all --- a situation that occurs, for example,
if the kernel $K$ is an independence sampler.
However, in rare event probability estimation,
the OIF is typically multiple multiple orders of magnitude, demonstrating major potential for WE to reduce MCMC's variance.
In Section \ref{sec:experiments},
we explicitly calculate the OIF for several examples.

We close this section by discussing the key lemma that 
allows us to optimize WE's variance and prove the sharpness result \eqref{eq:WE_upper_bd}.

\begin{lemma}{\label{lem:exact}}
Consider a WE scheme with a $V$-uniformly ergodic  kernel $K$.
Assume binned multinomial resampling is used, 
$\mu_0\left(V\right) < \infty$, and $\left\lVert f^2 \slash V \right\rVert < \infty$.
Then, as $T \rightarrow \infty$, WE estimates for $\mu\left(f\right)$ satisfy
 \begin{align}
    \label{eq:WE_var_expression}
    & \Var\left[\frac{1}{T}
    \sum_{t=0}^{T-1} \sum_{i=1}^N w_t^i f\left(\xi_t^i\right)\right] \\
    & = \frac{1}{T^2} \sum_{t=0}^{T-2}
    \E\left[\sum_u \frac{w_t\left(u\right)^2}{N_t\left(u\right)} \left(
    \Var_{\eta_t^u}\left[Kh_f\right]
    + \Var_{\eta_t^u}\left[v_f\right]
    + \eta_t^u\left(v_f\right)^2\right)\right]
    + \mathcal{O}\left(\frac{1}{T^2}\right),
 \end{align}
where $\eta_t^u = \frac{1}{w_t\left(u\right)} \sum_{i \in u} w_t^i \delta_{\xi_t^i}$
denotes the empirical distribution of particles in bin $u$.
\end{lemma}

This lemma decomposes WE's asymptotic variance into separate contributions from the different bins. 
The lemma reveals how these bins can be optimized to minimize WE's variance.
In our proof of \eqref{eq:WE_upper_bd}, we choose the bins and bin allocations in the following way:

\begin{enumerate}
    \item We first define a large number of spatial bins in the $K h_f$ and $v_f$ 
    coordinates, ensuring that most of the terms $\Var_{\eta_t^u}\left[Kh_f\right]
    + \Var_{\eta_t^u}\left[v_f\right]$ in the variance decomposition are small.
    \item We next minimize the $w_t\left(u\right)^2 \eta_t^u\left(v_f\right)^2 \slash N_t\left(u\right)$ terms in the variance decomposition by allocating particles to bins according to the rule
    \begin{equation}
    \label{eq:bin_rule}
        \frac{N_t\left(u\right)}{N} \approx \frac{w_t\left(u\right) \eta_t^u\left(v_f\right)}
        {\sum_{u^{\prime}} w_t\left(u^{\prime}\right)
        \eta_t^{u^{\prime}}\left(v_f\right)}.
    \end{equation}
    \item 
    As $\epsilon \rightarrow 0$, 
    we increase the number of particles and bins to ensure that WE's variance lies within a factor of $1 + \epsilon$ of the optimal variance $\mu\left(v_f\right)^2 \slash NT$.
\end{enumerate}

While this optimization strategy is convenient for proving the variance bound \eqref{eq:WE_upper_bd},
it would be difficult to carry out this strategy in WE applications. 
The main problem is that functions $Kh_f$ and $v_f$ are typically unknown.
As a more practical alternative, therefore, Aristoff and Zuckerman \cite{aristoff2020optimizing}
have developed an optimization approach for WE that uses coarse-grained approximations of the functions $K h_f$ and $v_f$.
We apply this optimization approach in all the numerical examples in Section \ref{sec:experiments}.

\section{Mathematical proofs}{\label{sec:theory}}

Here, we prove our theoretical results concerning the bias, convergence, and variance of WE estimates.

\subsection{Bias}

We first examine the bias of a splitting method's estimates.
As a central analysis tool,
we consider a filtration of $\sigma$-algebras
$\mathcal{F}_0 \subseteq \hat{\mathcal{F}_0} \subseteq \mathcal{F}_1 \subseteq
\hat{\mathcal{F}_1} \subseteq \cdots$
that
satisfy the following assumptions: 
\begin{assumptions}{\label{ass:filtration}}
~
\begin{enumerate}
\item[(i)] The variables $\left(\xi_t^i, w_t^i, C_t^i\right)_{1 \leq i \leq N_t}$ are measurable with respect to $\mathcal{F}_t$.
\item[(ii)] Conditional on $\mathcal{F}_t$, the copy numbers $N_t^1, \ldots, N_t^{N_t}$ each have mean $\E\left[\left.N_t^i\right|\mathcal{F}_t\right] = C_t^i$.
\item[(iii)] The variables $\left(\hat{\xi}_t^i, \hat{w}_t^i\right)_{1 \leq i \leq N_{t+1}}$
are measurable with respect to $\hat{\mathcal{F}}_t$.
\item[(iv)] Conditional on $\hat{\mathcal{F}}_t$,
the particles $\xi_{t+1}^1, \ldots, \xi_{t+1}^{N_{t+1}}$ are independent with
$\Law\left(\left.\xi_{t+1}^i\right|\hat{\mathcal{F}}_t\right) = K\left(\hat{\xi}_t^i, \cdot\right)$.
\end{enumerate}
\end{assumptions}
The filtration $\mathcal{F}_0 \subseteq \hat{\mathcal{F}_0} \subseteq \mathcal{F}_1 \subseteq
\hat{\mathcal{F}_1} \subseteq \cdots$ has a natural interpretation in terms of the information that is available at each step of the splitting method.
$\mathcal{F}_t$ contains all the information available after the identity of the particles $\xi_t^1, \ldots, \xi_t^{N_t}$ is revealed
and before the identities of the children particles
$\hat{\xi}_t^1, \ldots, \hat{\xi}_t^{N_{t+1}}$ are revealed.
$\hat{\mathcal{F}}_t$ contains all the information in $\mathcal{F}_t$ and also the identities of the children particles $\hat{\xi}_t^1, \ldots, \hat{\xi}_t^{N_{t+1}}$.

The $\sigma$-algebras are useful because they reveal a rich martingale structure that underlies splitting schemes, which was originally exploited by Del Moral in
\cite{del2012feynman}.
We introduce this martingale structure in the following lemma:
\begin{lemma}{\label{lem:martingale}}
Fix a time $T \geq 0$ and a function $f$ with $\mu_0 K^T \left|f\right| < \infty$.
Define
\begin{align}
    & M_t = \E\left[\left. \sum_{i=1}^{N_T} w_T^i f\left(\xi_T^i\right) \right| \mathcal{F}_t\right],
    & \hat{M}_t = \E\left[\left. \sum_{i=1}^{N_T} w_T^i f\left(\xi_T^i\right) \right| \hat{\mathcal{F}}_t\right],
    \quad 0 \leq t \leq T - 1.
\end{align}
Then, $M_0, \hat{M}_0, \ldots, M_{T-1}, \hat{M}_{T-1}$ is a martingale that satisfies
\begin{align}
    & M_t = \sum_{i=1}^{N_t} w_t^i \left(K^{T-t} f\right)\left(\xi_t^i\right),
    & \hat{M}_t = \sum_{i=1}^{N_{t+1}} \hat{w}_t^i \left(K^{T-t} f\right)\left(\hat{\xi}_t^i\right),
    \quad 0 \leq t \leq T-1.
\end{align}
\end{lemma}
\begin{proof}
Set $M_T = \sum_{i=1}^{N_T} w_T^i f\left(\xi_T^i\right)$
and assume for some $0 \leq t \leq T-1$ the representation $M_{t+1} = \sum_{i=1}^{N_{t+1}} w_{t+1}^i \left(K^{T-t-1} f\right) \left(\xi_{t+1}^i\right)$
is valid.
Then, using Assumption \ref{ass:filtration} (iv),
\begin{align}
    \hat{M}_t &= \E\left[\left.M_{t+1}\right|\hat{\mathcal{F}}_t\right] \\
    &= \sum_{i=1}^{N_{t+1}} \E\left[\left. w_{t+1}^i \left(K^{T-t-1} f\right)\left(\xi_{t+1}^i\right) \right| \mathcal{F}_t\right] \\
    &= \sum_{i=1}^{N_{t+1}} \hat{w}_t^i 
    \left(K^{T-t} f\right)\left(\hat{\xi}_t^i\right).
\end{align}
Using Assumption \ref{ass:filtration} (ii) and the fact that children of $\xi_t^i$ receive weights $w_t^i \slash C_t^i$,
\begin{align}
    M_t &= \E\left[\left. \hat{M}_t \right| \mathcal{F}_t\right] \\
    &= \E\left[\left. \sum_{i=1}^{N_{t+1}} \hat{w}_t^i 
    \left(K^{T-t} f\right)\left(\hat{\xi}_t^i\right) \right| \mathcal{F}_t \right] \\
    &= \E\left[\left. \sum_{i=1}^{N_t} N_t^i \frac{w_t^i}{C_t^i}
    \left(K^{T-t} f\right)\left(\xi_t^i\right) \right| \mathcal{F}_t\right] \\
    &= \sum_{i=1}^{N_t} w_t^i 
    \left(K^{T-t} f\right)\left(\xi_t^i\right).
\end{align}
\end{proof}

Lemma \ref{lem:martingale} allows us to prove the following generalization of Proposition \ref{prop:unbiased}, which simultaneously establishes bias properties for all splitting methods and WE methods.

\begin{proposition}
Consider a splitting method with a Harris ergodic kernel $K$, and assume $f$ is bounded.
The estimates for $\mu\left(f\right)$ have the following bias properties:
\begin{enumerate}
	\item For any $\tau \geq 0$ and any $T \geq 0$, the estimate
	\begin{equation}
	    \frac{1}{T} \sum_{t=\tau}^{\tau+T-1} \sum_{i=1}^{N_t} w_t^i f\left(\xi_t^i\right)
	\end{equation}
    has the same expectation as the trajectory average 
    \begin{equation}
    \frac{1}{T} \sum_{t=\tau}^{\tau+T-1} f\left(X_t\right),
    \end{equation}
	where $X_t$ is a Markov chain with transition kernel $K$
	and initial distribution $\mu_0$.
	\item The estimates for $\mu\left(f\right)$ are asymptotically unbiased in the limit as $\tau + T \rightarrow \infty$:
	\begin{equation}
	    \E\left[\frac{1}{T} \sum_{t=\tau}^{\tau+T-1} \sum_{i=1}^{N_t} w_t^i f\left(\xi_t^i\right)\right]
	    \stackrel{\tau + T \rightarrow \infty}{\rightarrow} 
	    \mu\left(f\right).
	\end{equation}
\end{enumerate}
\end{proposition}
\begin{proof}
Using Lemma \ref{lem:martingale}, we calculate
\begin{equation}
\E\left[\frac{1}{T} \sum_{t=\tau}^{\tau + T - 1} \sum_{i=1}^{N_t} w_t^i f\left(\xi_t^i\right)\right]
= \frac{1}{T} \sum_{t=\tau}^{T + \tau - 1} \mu_0 K^t f.
\end{equation}
As a consequence of Harris ergodicity, we have the convergence $\left\lVert \mu_0 K^t - \mu \right\rVert \stackrel{t \rightarrow \infty}{\rightarrow} 0$ \cite[ch. 13]{meyn2012markov}.
Sending $\tau + T \rightarrow \infty$, we verify
\begin{equation}
\left| \E\left[\frac{1}{T} \sum_{t=\tau}^{\tau + T - 1} \sum_{i=1}^{N_t} w_t^i f\left(\xi_t^i\right)\right]
- \mu\left(f\right) \right|
\leq 
\frac{\left\lVert f \right\rVert }{T} \sum_{t=\tau}^{\tau + T - 1}
\left\lVert \mu_0 K^t - \mu \right\rVert \rightarrow 0.
\end{equation}
\end{proof}

\subsection{Convergence}

In this section, we prove that a splitting method provides asymptotically
consistent estimates if and only if
the sum of the weights is almost surely one.
To prove this result, we observe that the splitting method defined in
Algorithm \ref{alg:splitting} ensures that
the sum of the weights has expected value one at all times $t \geq 0$.
Moreover, the sum of the weights $\sum_{i=1}^{N_t} w_t^i$ is a nonnegative martingale,
and a nonnegative martingale must converge 
with probability one as $t \rightarrow \infty$ \cite{kallenberg2006foundations}.
This observation immediately verifies the result in Proposition \ref{prop:collapse}.
To prove Theorem \ref{thm:requirement},
we also need the following lemma:

\begin{lemma}\label{lem:sqrtV_geomerg}
If $K$ is $V$-uniformly ergodic, then $K$ is also $\sqrt{V}$-uniformly ergodic.
\end{lemma}
\begin{proof}
 By Jensen's inequality, for any positive measure $\eta$,
    \begin{align}
      \sup_{\left|g\right| \leq \sqrt{V}} \eta\left(\left|g\right|\right) 
      & \leq \left\lVert \eta \right\rVert 
      \sup_{g^2 \leq V} \frac{\eta}{\left\lVert \eta \right\rVert}
      (\left|g\right|) \\
      & \leq \left\lVert \eta \right\rVert
      \sup_{g^2 \leq V} \sqrt{\frac{\eta}{\left\lVert \eta \right\rVert}\left(g^2\right)} \\
      & = \sqrt{\left\lVert \eta \right\rVert }
      \sqrt{\sup_{\left|g\right| \leq V} \eta\left(\left|g\right|\right)}.
    \end{align}
    Taking $\eta = \left|K^t\left(x,\cdot\right) - \mu\right|$ and applying $V$-uniform ergodicity gives the desired result.
\end{proof}

\begin{proof}[Proof of Theorem \ref{thm:requirement}]
First, we observe that (iii) implies (i).

Next, we show that (i) implies (ii).
Part (i) indicates the convergence in probability
$\frac{1}{T} \sum_{t=0}^{T-1} \sum_{i=1}^{N_t} w_t^i \stackrel{T \rightarrow \infty}{\rightarrow} 1$.
Since $\sum_{i=1}^{N_t} w_t^i$ is a nonnegative martingale, 
$\sum_{i=1}^{N_t} w_t^i$ 
converges almost surely to a random variable $W_{\infty}$ as $t \rightarrow \infty$.
Hence, we must have $W_{\infty} = 1$.
Next, for fixed $t \geq 0$, Fatou's lemma implies
\begin{equation}
\sum_{i=1}^{N_t} w_t^i
= \liminf_{T \rightarrow \infty}
\E\left[\left. \sum_{i=1}^{N_T} w_T^i \right| \mathcal{F}_t\right]
\geq \E\left[\left. \liminf_{T \rightarrow \infty} \sum_{i=1}^{N_T} w_T^i \right| \mathcal{F}_t\right]
= \E\left[\left. W_{\infty} \right| \mathcal{F}_t\right]
= 1.
\end{equation}
Since $\E\left[\sum_{i=1}^{N_t} w_t^i\right] = 1$ and $\sum_{i=1}^{N_t} w_t^i \geq 1$,
we conclude that $\sum_{i=1}^{N_t} w_t^i = 1$ with probability one.
Since $t \geq 0$ is arbitrary, we have verified (ii).

Last of all, we prove that (ii) implies (iii).
We assume without loss of generality $f \geq 0$, and we show that almost surely
\begin{equation}
\label{eq:simpler}
    \Prob\left\{\left.\frac{1}{T} \sum_{t=0}^{T-1} \sum_{i=1}^{N_t} w_t^i f\left(\xi_t^i\right)
    \stackrel{T \rightarrow \infty}{\rightarrow}
    \int \mu\left(\mathop{dx}\right) f\left(x\right)\right| \mathcal{F}_0 \right\} = 1.
\end{equation}
We fix $T \geq 0$ and compute the conditional variance
\begin{equation}
    \Var\left[\left. \sum_{t=0}^{T-1} \sum_{i=1}^{N_t} w_t^i f\left(\xi_t^i\right)\right|\mathcal{F}_0\right] 
    = \sum_{s,t=0}^{T-1} \Cov\left[\left.\sum_{i=1}^{N_s} w_s^i f\left(\xi_s^i\right),
    \sum_{i=1}^{N_t} w_t^i f\left(\xi_t^i\right)\right|\mathcal{F}_0\right].
\end{equation}
For $s \leq t$, the conditional covariance terms satisfy
\begin{align}
    & \Cov\left[\left.\sum_{i=1}^{N_s} w_s^i f\left(\xi_s^i\right),
    \sum_{i=1}^{N_t} w_t^i f\left(\xi_t^i\right)\right|\mathcal{F}_0\right] \\
    &= \Cov\left[\left.\sum_{i=1}^{N_s} w_s^i f\left(\xi_s^i\right),
    \sum_{i=1}^{N_s} w_s^i K^{t-s} f\left(\xi_s^i\right)\right|\mathcal{F}_0\right] \\
    &\leq \Var\left[\left.\sum_{i=1}^{N_s} w_s^i f\left(\xi_s^i\right)\right|\mathcal{F}_0\right]^{1 \slash 2}
    \Var\left[\left.\sum_{i=1}^{N_s} w_s^i K^{t-s} f\left(\xi_s^i\right)\right|\mathcal{F}_0\right]^{1 \slash 2}.
\end{align}
Using the fact that $\sum_{i=1}^{N_s} w_s^i = 1$, we calculate
\begin{align}
    & \Var\left[\left.\sum_{i=1}^{N_s} w_s^i K^{t-s} f\left(\xi_s^i\right)\right|\mathcal{F}_0\right] \\
    & \leq \E\left[\left.\left| \sum_{i=1}^{N_s} w_s^i \left(K^{t-s} f - \mu\left(f\right)\right)\left(\xi_s^i\right)
    \right|^2 \right|\mathcal{F}_0\right] \\
    &\leq \E\left[\left. \sum_{i=1}^{N_s} w_s^i \left(K^{t-s} f - \mu\left(f\right)\right)^2\left(\xi_s^i\right)
    \right|\mathcal{F}_0\right] \\
    &= \frac{1}{N_0} \sum_{i=1}^{N_0}
    K^s\left(\left(K^{t-s} f - \mu\left(f\right)\right)^2\right)\left(\xi_0^i\right)
\end{align}
Using the $\sqrt{V}$-uniform ergodicity and $V$-uniform ergodicity of $K$, the last term is size $\mathcal{O}\left(r^{-\left(t-s\right)}\right)$ for a fixed constant $r > 1$,
and we obtain a bound of the form
\begin{equation}
    \Var\left[\left.\frac{1}{T} \sum_{t=0}^{T-1} \sum_{i=1}^{N_t} w_t^i f\left(\xi_t^i\right)\right| \mathcal{F}_0\right] 
    \leq \frac{C \left\lVert f^2 \slash V \right\rVert}{T},
\end{equation}
where $C$ is independent of $T$ and $f$.
Since the conditional variance terms are summable for $T = 1, 4, 9, \ldots$,
the WE estimates converge by a Borel-Cantelli argument, and we find
\begin{equation}
    \lim_{T \rightarrow \infty} \frac{1}{T^2} \sum_{t=0}^{T^2-1} \sum_{i=1}^{N_t} w_t^i f\left(\xi_t^i\right)
    = \lim_{T \rightarrow \infty} \E\left[\left.\frac{1}{T^2} \sum_{t=0}^{T^2-1} \sum_{i=1}^{N_t} w_t^i f\left(\xi_t^i\right)\right| \mathcal{F}_0\right]
    = \mu\left(f\right),
\end{equation}
with conditional probability one.
We can strengthen the almost sure convergence for $T = 1, 4, 9, \ldots$
to almost sure convergence for 
$T = 1, 2, 3, \ldots$ 
by noting that 
\begin{align}
\frac{T^2}{T^2 + s}
\left(\frac{1}{T^2} \sum_{t=0}^{T^2-1} \sum_{i=1}^{N_t} w_t^i f\left(\xi_t^i\right)\right)
&\leq
\frac{1}{T^2 + s} \sum_{t=0}^{T^2 + s-1} \sum_{i=1}^{N_t} w_t^i f\left(\xi_t^i\right) \\
&\leq
\frac{\left(T+1\right)^2}{T^2 + s}
\left(\frac{1}{\left(T+1\right)^2} \sum_{t=0}^{\left(T+1\right)^2-1} \sum_{i=1}^{N_t} w_t^i f\left(\xi_t^i\right)\right)
\end{align}
whenever $T^2 \leq T^2 + s \leq \left(T+1\right)^2$.
Hence, we verify equation \eqref{eq:simpler}, completing the proof.
\end{proof}

Lastly, we verify that any conditionally independent resampling scheme
that maintains a sum of weights equal to one is a WE scheme:

\begin{proof}[Proof of Proposition \ref{prop:uniqueness}]
We condition on the matrix $\bm{P} \in \mathbb{R}^{M \times N_t}$ and on the locations and weights of the parents.
Before removing the particles with zero weights,
the weights $\hat{w}_t^1, \ldots, \hat{w}_t^M$ are independent.
Since $\sum_{i=1}^M \hat{w}_t^i = 1$,
we find
\begin{equation}
    0
    = \Var\left[\sum_{i=1}^M \hat{w}_t^i\right]
    = \sum_{i=1}^M \Var\left[\hat{w}_t^i\right].
\end{equation}
and each weight $\hat{w}_t^i$ is constant with probability one.
Hence, we can define bins $u_c$ consisting of all the parents whose children receive weights $\hat{w}_t^i = c$.
There is a fixed number of children per bin and all the children receive the same weight, so the splitting method is a WE method.
\end{proof}

\subsection{Variance}{\label{sec:var_analysis}}

In this final subsection of technical results, 
we bound the variance of WE estimates. 
Our main approach, following the analysis developed by Del Moral \cite{del2012feynman},
is to decompose the variance of WE estimates
as a sum of squared martingale differences
and then manipulate the martingale difference terms to obtain sharp error bounds.

The martingale we use is slightly different from the one described in Lemma \ref{lem:martingale}, since we need to account for the time-averaging that produces WE estimates.
The following lemma introduces this martingale and gives an explicit formula for the martingale differences:
\begin{lemma}{\label{lem:martingale2}}
Fix $T \geq 0$ and a function $f$ with $\mu_0 K^t \left|f\right| < \infty$ for $0 \leq t \leq T-2$.
Define
\begin{align}
    & Y_t = \E\left[\left. \frac{1}{T} \sum_{t=0}^{t-1} \sum_{i=1}^N w_T^i f\left(\xi_T^i\right) \right| \mathcal{F}_t\right],
    & \hat{Y}_t = \E\left[\left. \frac{1}{T} \sum_{t=0}^{T-1} \sum_{i=1}^N w_T^i f\left(\xi_T^i\right) \right| \hat{\mathcal{F}}_t\right].
\end{align}
Then, $Y_0, \hat{Y}_0, \ldots, Y_{T-1}, \hat{Y}_{T-2}$ is a martingale with martingale differences given by
\begin{align}
\label{eq:formula}
    & \hat{Y}_t - Y_t = \frac{1}{T} \left[\sum_{i=1}^N \hat{w}_t^i K h_{t+1}^T\left(\hat{\xi}_t^i\right) - \sum_{i=1}^N w_t^i K h_{t+1}^T\left(\xi_t^i\right)\right], \\
    & Y_{t+1} - \hat{Y}_t = \frac{1}{T} \left[ \sum_{i=1}^N \hat{w}_t^i \left(h_{t+1}^T\left(\xi_{t+1}^i\right) - Kh_{t+1}^T\left(\hat{\xi}_t^i\right)\right)\right],
\end{align}
where we have introduced shorthand 
$h_t^T = \sum_{s=t}^{T} K^{s-t} \left(f - \mu\left(f\right)\right)$.
\end{lemma}
\begin{proof}
Use Lemma \ref{lem:martingale} and simplify terms.
\end{proof}

Using the martingale in Lemma \ref{lem:martingale2}
we can prove the following lower bound on WE's asymptotic variance:

\begin{proposition}\label{prop:LB}
Consider a WE scheme with a kernel $K$ that is geometrically ergodic and $V$-uniformly ergodic, with $\left\lVert f^2 \slash V \right\rVert < \infty$.
The variance of WE estimates for $\mu\left(f\right)$ is bounded from below by
\begin{equation}
\label{eq:lower}
	\liminf_{T \rightarrow \infty} T \Var\left[\frac{1}{T} \sum_{t=0}^{T-1} \sum_{i=1}^N w_t^i f\left(\xi_t^i\right)\right] \geq \frac{\mu\left(v_f\right)^2}{N}.
\end{equation}
\end{proposition}
\begin{proof}
First, the $V$-uniform ergodicity and $\sqrt{V}$-uniform ergodicity of $K$ guarantee $\left\lVert v_f^2 \slash V \right\rVert < \infty$.
Hence, the right-hand side is finite and we can consider without loss of generality a subsequence of $T$ values for which
$\Var\left[\frac{1}{T} \sum_{t=0}^{T-1} \sum_{i=1}^N w_t^i f\left(\xi_t^i\right)\right] < \infty$.
Then a martingale variance decomposition using Lemma \ref{lem:martingale2} guarantees
\begin{equation}
    \Var\left[\frac{1}{T}\sum_{t=0}^{T-1}\sum_{i=1}^N w_t^i f\left(\xi_t^i\right)\right] 
    \geq \sum_{t=0}^{T-2} \E\left|Y_{t+1}-\hat{Y_t}\right|^2
    = \sum_{t=0}^{T-2} \E\left[\Var\left[\left. Y_{t+1} \right| \hat{\mathcal{F}}_t\right]\right].
\end{equation}
Applying Jensen's inequality and setting $v_t^T\left(x\right) = \Var_{K\left(x,\cdot\right)}\left[h_{t+1}^T\right]^{1 \slash 2}$, we calculate
\begin{align}
    T^2 \E\left[\Var\left[\left. Y_{t+1} \right| \hat{\mathcal{F}}_t\right]\right]
    &= \E\left[\Var\left[\left.\sum_{i=1}^N w_{t+1}^i v_{t+1}^T\left(\xi_{t+1}^i\right)\right|\hat{\mathcal{F}}_t\right]\right] \\
    &= \E\left[\sum_{i=1}^N \left| \hat{w}_t^i v_t^T\left(\hat{\xi}_t^i\right)\right|^2\right] \\
    &\geq \frac{1}{N} 
    \E\left|\sum_{i=1}^N \hat{w}_t^i v_t^T\left(\hat{\xi}_t^i\right)\right|^2 \\
    &\geq \frac{\E\left[\sum_{i=1}^N \hat{w}_t^i v_t^T\left(\hat{\xi}_t^i\right)\right]^2}{N} \\
    &= \frac{\mu_0\left(K^t v_t^T\right)^2}{N}.
\end{align}
In summary, we find
\begin{align}
    T \Var\left[\frac{1}{T}\sum_{t=0}^{T-1}\sum_{i=1}^N w_t^i f\left(\xi_t^i\right)\right]
    &\geq 
    \frac{1}{TN} \sum_{t=0}^{T-1} \mu_0\left(K^t v_t^T\right)^2 \\
    &= \frac{1}{N} \int_0^1
    \mu_0\left(K^{\left\lfloor sT \right\rfloor} v_{\left\lfloor sT \right\rfloor}^T\right)^2 \mathop{ds}.
\end{align}
For any $0 < s \leq 1$, we observe that
$\left\lVert \mu_0 K^{\left\lfloor sT \right\rfloor} - \mu\right\rVert \stackrel{T \rightarrow \infty}{\rightarrow} 0$
and also $v^T_{\left\lfloor sT \right\rfloor} \stackrel{T \rightarrow \infty}{\rightarrow} v_f$ pointwise on the set $\left\{V < \infty\right\}$.
Hence, by a useful generalization of Fatou's lemma (see \cite[sec. 11.4]{royden1988real}), 
\begin{equation}
    \liminf_{T \rightarrow \infty}
    \mu_0\left(K^{\left\lfloor sT \right\rfloor} v_{\left\lfloor sT \right\rfloor}^T\right)
    \geq \mu\left(v_f\right).
\end{equation}
We are able to conclude
\begin{equation}
    \liminf_{T \rightarrow \infty} T \Var\left[\frac{1}{T}\sum_{t=0}^{T-1}\sum_{i=1}^N w_t^i f\left(\xi_t^i\right)\right]
    \geq \frac{\mu\left(v_f\right)^2}{N}.
\end{equation}
\end{proof}

Throughout our variance analysis, we have made the minimal assumptions that are needed to prove our results.
In Proposition \ref{prop:LB}, we needed the assumption $\left\lVert f^2 \slash V \right\rVert < \infty$ to ensure that the variance function $v_f^2$ is well-defined on a set of full $\mu$ measure and $\mu\left(v_f\right) < \infty$.
Moving forward, in order to prove Lemma \ref{lem:exact},
we also need the assumption
$\mu_0\left(V\right) < \infty$.
This condition rules out a degenerate situation where the initial particles are drawn so far out of equilibrium
that there is a lingering effect on the first and second moments of WE estimates ---
the same assumption would also be needed to bound the variance of direct MCMC estimates as well.

As we demonstrate below, our minimal assumptions are enough to verify Lemma \ref{lem:exact}, which gives a precise expression for WE's asymptotic variance.

\begin{proof}[Proof of Lemma \ref{lem:exact}]
We manipulate the martingale differences in Lemma \ref{lem:martingale2} to find
\begin{align}
    T^2 \E\left| Y_{t+1} - Y_t \right|^2
    &= T^2 \E\left[\Var\left[\left. Y_{t+1} \right| \mathcal{F}_t\right]\right] \\
    &= \E\left[\Var\left[\left.\sum_{i=1}^N w_{t+1}^i h_{t+1}^T\left(\xi_{t+1}^i\right)\right|\mathcal{F}_t\right]\right] \\
    &= \E\left[\sum_{i=1}^N \left|w_{t+1}^i\right|^2 \Var\left[\left. h_{t+1}^T\left(\xi_{t+1}^i\right)\right|\mathcal{F}_t\right]\right] \\
    &= \E\left[\sum_u \frac{w_t\left(u\right)^2}{N_t\left(u\right)} \Var_{\eta_t^u K}\left[h_{t+1}^T\right]\right],
\end{align}
where we have used the definition of binned multinomial resampling
and we have set $\eta_t^u = \frac{1}{w_t\left(u\right)} \sum_{i \in u} w_t^i \delta\left(\xi_t^i\right)$.
Hence,
\begin{align}
    \Var\left[\frac{1}{T}\sum_{t=0}^{T-1}\sum_{i=1}^N w_t^i f\left(\xi_t^i\right)\right]
    = \frac{\Var_{\mu_0}\left[h_0^T\right]}{N T^2} 
    + \frac{1}{T^2} \sum_{t=0}^{T-2}
    \E\left[\sum_u \frac{w_t\left(u\right)^2}{N_t\left(u\right)} \Var_{\eta_t^u K}\left[h_{t+1}^T\right] \right].
\end{align}
In this decomposition, the $\sqrt{V}$-uniform ergodicity of $K$
and the condition $\mu_0\left(V\right) < \infty$
guarantee the first term is asymptotically $\mathcal{O}\left(T^{-2}\right)$.
To analyze the second term, we first calculate
\begin{align}
    & \left|
    \E\left[\sum_u \frac{w_t\left(u\right)^2}{N_t\left(u\right)} \left(\Var_{\eta_t^u K}\left[h_{t+1}^T\right]
    - \Var_{\eta_t^u K}\left[h_f\right]\right) \right]\right| \\
    & \leq
    \E\left[\sum_u w_t\left(u\right) \left|\Var_{\eta_t^u K}\left[h_{t+1}^T\right]
    - \Var_{\eta_t^u K}\left[h_f\right]\right| \right] \\
    & \leq
    \E\left[\sum_u w_t\left(u\right) \Var_{\eta_t^u K}\left[h_{t+1}^T + h_f\right] \right]^{1 \slash 2} \E\left[\sum_u w_t\left(u\right) \Var_{\eta_t^u K}\left[h_{t+1}^T - h_f\right] \right]^{1 \slash 2}
    \\
    & \leq
    \E\left[\sum_u w_t\left(u\right) \eta_t^u K\left|h_{t+1}^T + h_f\right|^2 \right]^{1 \slash 2} \E\left[\sum_u w_t\left(u\right) \eta_t^u K\left|h_{t+1}^T - h_f\right|^2 \right]^{1 \slash 2}
    \\
    & = \left(\mu_0 K^{t+1} \left|h_{t+1}^T + h_f\right|^2\right)^{1 \slash 2} 
    \left(\mu_0 K^{t+1} \left|h_{t-1}^T - h_f\right|^2\right)^{1 \slash 2}.
\end{align}
This leads to the bound
\begin{align}
    & \left|\frac{1}{T^2} \sum_{t=0}^{T-2}
    \E\left[\sum_u \frac{w_t\left(u\right)^2}{N_t\left(u\right)} \Var_{\eta_t^u K}\left[h_{t+1}^T\right] \right]
    - \frac{1}{T^2} \sum_{t=0}^{T-2}
    \E\left[\sum_u \frac{w_t\left(u\right)^2}{N_t\left(u\right)} \Var_{\eta_t^u K}\left[h_f\right] \right]\right| \\
    & \leq \frac{1}{T^2} \sum_{t=0}^{T-2} \left(\mu_0 K^{t+1} \left|h_{t+1}^T + h_f\right|^2\right)^{1 \slash 2} 
    \left(\mu_0 K^{t+1} \left|h_{t+1}^T - h_f\right|^2\right)^{1 \slash 2}.
\end{align}
The $\sqrt{V}$-uniform ergodicity of $K$, the $V$-uniform ergodicity of $K$, and the condition $\mu_0\left(V\right) < \infty$ guarantee that the last quantity is $\mathcal{O}\left(T^{-2}\right)$ as $T \rightarrow \infty$, confirming the result.
\end{proof}

As the last step in our technical analysis, we use Lemma \ref{lem:exact} to construct a WE scheme that nearly achieves the optimal variance bound.

\begin{proposition}

Consider a WE scheme with a kernel $K$ that is geometrically ergodic and $V$-uniformly ergodic, 
with $\left\lVert f^2 \slash V \right\rVert < \infty$.
Then, for any $\epsilon > 0$,
there is a WE scheme that satisfies
\begin{equation}
\label{eq:explicit}
\limsup_{T \rightarrow \infty} T \Var\left[\frac{1}{T} \sum_{t=0}^{T-1} \sum_{i=1}^N w_t^i f\left(\xi_t^i\right)\right] \leq \left(1 + \epsilon\right) \frac{\mu\left(v_f\right)^2}{N}
\end{equation} 
if the number of particles $N$ is sufficiently large.
\end{proposition}

\begin{proof}
In the case $\mu\left(v_f\right) = 0$, we must also have $\mu\left(v_f^2\right) = 0$, and direct MCMC sampling is sufficient to achieve the asymptotic variance upper bound.
Next, we consider the case $\mu\left(v_f\right) > 0$.
We assume initial particles are drawn from a distribution satisfying $\mu_0\left(V\right) < \infty$.
We define bins based on spatial sets
\begin{align}
    & u_{i,j} = \left\{x \in X\colon i - \frac{1}{2} < \frac{Kh_f\left(x\right)}{\delta} \leq i + \frac{1}{2}, \quad j - \frac{1}{2} < \frac{v_f\left(x\right)}{\delta} \leq j + \frac{1}{2} \right\}, \\
    & u_{\infty} = X \setminus \left(u_{i,j}\right)_{-J \leq i,j \leq J}
\end{align}
where $\delta$ and $J$ are parameters to be tuned.
Here, in a slight abuse of notation, we are using $u_{i,j}$ to refer both to a spatial set and to the indices of the particles in that set.
We set bin allocations $N_t\left(u\right)$ to satisfy
\begin{equation}
\label{eq:particle_alloc}
\frac{N_t\left(u\right)}{N} \geq \max\left\{
\delta w_t\left(u\right),
\,\left(1-2\delta\right)
\frac{w_t\left(u\right)\eta_t^u\left(v_f\right)}
{\sum_u w_t\left(u\right) \eta_u^u\left(v_f\right)}\right\},
\end{equation} 
which is always possible when the number of particles $N$ is sufficiently large.

Having introduced an explicit WE scheme, we bound its asymptotic variance using Lemma \ref{lem:exact}.
We perform the following three-step variance calculation:

\emph{Step 1.} We bound the intrabin variance in the $K h_f$ and $v_f$ coordinates using
\begin{align}
    & \frac{1}{T} \sum_{t=0}^{T-2}
    \E\left[\sum_u \frac{w_t\left(u\right)^2}{N_t\left(u\right)} \left(
    \Var_{\eta_t^u}\left[Kh_f\right]
    + \Var_{\eta_t^u}\left[v_f\right]\right)\right] \\
    & \leq \frac{\delta}{2 N} +
    \frac{1}{\delta T N} \sum_{t=0}^{T-2}
    \E\left[w_t\left(u_\infty\right) \left(
    \Var_{\eta_t^{u_\infty}}\left[Kh_f\right]
    + \Var_{\eta_t^{u_\infty}}\left[v_f\right]\right)\right] \\
    & \leq \frac{\delta}{2 N} +
    \frac{1}{\delta T N} \sum_{t=0}^{T-2}
    \E\left[w_t\left(u_\infty\right) \eta_t^{u_\infty}\left(
    \left|Kh_f\right|^2 + \left|v_f\right|^2\right)\right] \\
    & \leq \frac{\delta}{2 N} +
    \frac{\mu\left( \mathds{1}_{u_{\infty}} \left(\left(K h_f\right)^2 + v_f^2\right) \right)}{\delta N}
\end{align}

\emph{Step 2}. We bound the remaining asymptotic variance term by using
\begin{align}
    & \frac{1}{T} \sum_{t=0}^{T-2}
    \E\left[\sum_u \frac{\left|w_t\left(u\right) \eta_t^u\left(v_f\right)\right|^2}{N_t\left(u\right)} \right] \\
    &\leq \frac{1}{\left(1 - 2 \delta \right) N T} \sum_{t=0}^{T-2}
    \E\left| \sum_{i=1}^N w_t^i v_f\left(\xi_t^i\right)\right|^2 \\
    &= \frac{1}{\left(1 - 2\delta\right) NT} \sum_{t=0}^{T-2}
    \left(\mu\left(v_f\right)^2 + \Var\left[\sum_{i=1}^N w_t^i v_f\left(\xi_t^i\right)\right]\right).
\end{align}

\emph{Step 3}. To bound a quantity $\Var\left[\sum_{i=1}^N w_T^i v_f\left(\xi_T^i\right)\right]$, 
we consider the martingale $M_t$ that was introduced in Lemma \ref{lem:martingale}.
Using the function $v_f$ in place of $f$, Lemma \ref{lem:martingale} yields:
\begin{align}
    & \Var\left[\sum_{i=1}^N w_T^i v_f\left(\xi_T^i\right)\right] \\
    &= \Var\left[\frac{1}{N} \sum_{i=1}^N K^T v_f\left(\xi_0^i\right)\right] 
    + \sum_{t=0}^{T-1} \E\left[\Var\left[\left. \sum_{i=1}^N w_{t+1}^i K^{T-t-1} v_f\left(\xi_t^i\right) \right| \mathcal{F}_t\right]\right] \\
    &= \frac{1}{N} \Var_{\mu}\left[K^T v_f\right] 
    + \sum_{t=0}^{T-1} \E\left[\sum_u \frac{w_t\left(u\right)^2}{N_t\left(u\right)} \Var_{\eta_t^u K} \left[K^{T-t-1} v_f\right]\right] \\
    &\leq \frac{1}{\delta N} \Var_{\mu}\left[K^T v_f\right]
    + \frac{1}{\delta N} \sum_{t=0}^{T-1}  \E\left[\sum_{i=1}^N w_t^i \Var_{K\left(\xi_t^i, \cdot\right)} \left[K^{T-t-1} v_f\right]\right] \\
    &= \frac{1}{\delta N} \sum_{t=0}^{T} \Var_{\mu} \left[K^t v_f\right] \\
    &\leq \frac{1}{\delta N} \sum_{t=0}^{\infty} \Var_{\mu} \left[K^t v_f\right].
\end{align}
We confirm this last term is finite, because 
$\left\lVert v_f \slash \sqrt{V} \right\rVert < \infty$ and $K$ is $\sqrt{V}$-uniformly ergodic.

In summary, steps 1-3 reveal that
\begin{align}
\frac{1}{T} \sum_{t=0}^{T-2}
    & \E\left[\sum_u \frac{w_t\left(u\right)^2}{N_t\left(u\right)} \left(
    \Var_{\eta_t^u}\left[Kh\right]
    + \Var_{\eta_t^u}\left[v_f\right]
    + \eta_t^u\left(v_f\right)^2\right)\right] \\
    &\leq \frac{\delta}{2 N} +
    \frac{\mu\left( \mathds{1}_{u_{\infty}} \left(\left(K h_f\right)^2 + v_f^2\right) \right)}{\delta N} + \frac{\mu\left(v_f\right)^2}{\left(1 - 2 \delta \right) N} 
    + \sum_{t=0}^{\infty} \frac{\Var_{\mu} \left[K^t v_f\right]}{\left(\delta - 2 \delta^2\right) N^2}
\end{align}
By taking $\delta$ appropriately small 
and then taking $J$ and $N$ appropriately large, 
we can make this last quantity less than $\left(1 + \epsilon\right) \mu\left(v_f\right)^2 \slash N$,
thereby completing the proof.
\end{proof}

\section{Numerical experiments}{\label{sec:experiments}}

In this section, we apply WE
to compute rare event probabilities in three example problems.
These numerical experiments validate our formulas for WE's optimal variance while also demonstrating
the major potential for efficiency gains by using WE instead of MCMC.

\subsection{Geometric tail probabilities}{\label{sec:geometric}}

In the first example, 
our goal is estimating tail 
probabilities $\mu\left[a, \infty\right)$
for the geometric distribution
\begin{equation}
    \mu\left(x\right) = 2^{-x-1}, \quad x \in \mathbb{Z}^+ = \left\{0, 1, \ldots\right\}.
\end{equation}
To sample from $\mu$, we use a Markov chain with transition probabilities
\begin{equation}
\label{eq:markov}
    P\left(x, x + 1\right) = P\left(x, 0\right) = \frac{1}{2}.
\end{equation}
When $a$ is large, it would be very costly to estimate
tail probabilities $\mu\left[a, \infty\right) = 2^{-a}$ by direct MCMC sampling.
However, we show that WE can make these calculations more tractable.

\subsubsection{WE implementation}

In our numerical experiments, we use WE to estimate
the tail probability $\mu\left[a, \infty\right) = 2^{-a}$ for $a = 25$.
We draw initial particles from $\mu$,
and we sample for $T = 1000$ time steps.
Following the optimization strategy discussed in Section \ref{sub:variance},
we sort the particles into spatial bins based on the sets
\begin{equation}
    u_i = \left\{i\right\}, \quad 0 \leq i \leq 23, \quad u_{24} = \left[24, \infty\right),
\end{equation}
which are the exact level sets of $K h_f$.
Then, we allocate children particles to each bin according to the rule
\begin{equation}
    \frac{N_t\left(u\right)}{N} \approx \frac{w_t\left(u\right) \eta_t^u\left(v_f\right)}
    {\sum_{u^{\prime}} w_t\left(u^{\prime}\right) \eta_t^{u^{\prime}}\left(v_f\right)}.
\end{equation}

\subsubsection{WE results}

In Figure \ref{fig:figure2} below,
we present WE's relative variance constant
\begin{equation}
\textup{Relative Variance Constant} = \frac{NT}{\mu\left(f\right)^2} \Var\left[\frac{1}{T} \sum_{t=0}^{T-1} \sum_{i=1}^N w_t^i f\left(\xi_t^i\right)\right],
\end{equation}
calculated over $10^6$ independent trials for the function $f\left(x\right) = \mathds{1}\left\{x \geq 25\right\}$.
Additionally, we present theoretical relative variance constants
for MCMC and WE, 
calculated using the asymptotic theory developed in Section \ref{sec:variance}.
With just $N = \mathcal{O}\left(a\right)$ particles,
we find that
WE very nearly achieves the theoretical optimal variance,
thereby improving MCMC's variance by more than five orders of magnitude.

\begin{figure}[h!]
    \includegraphics[scale=.4, clip]{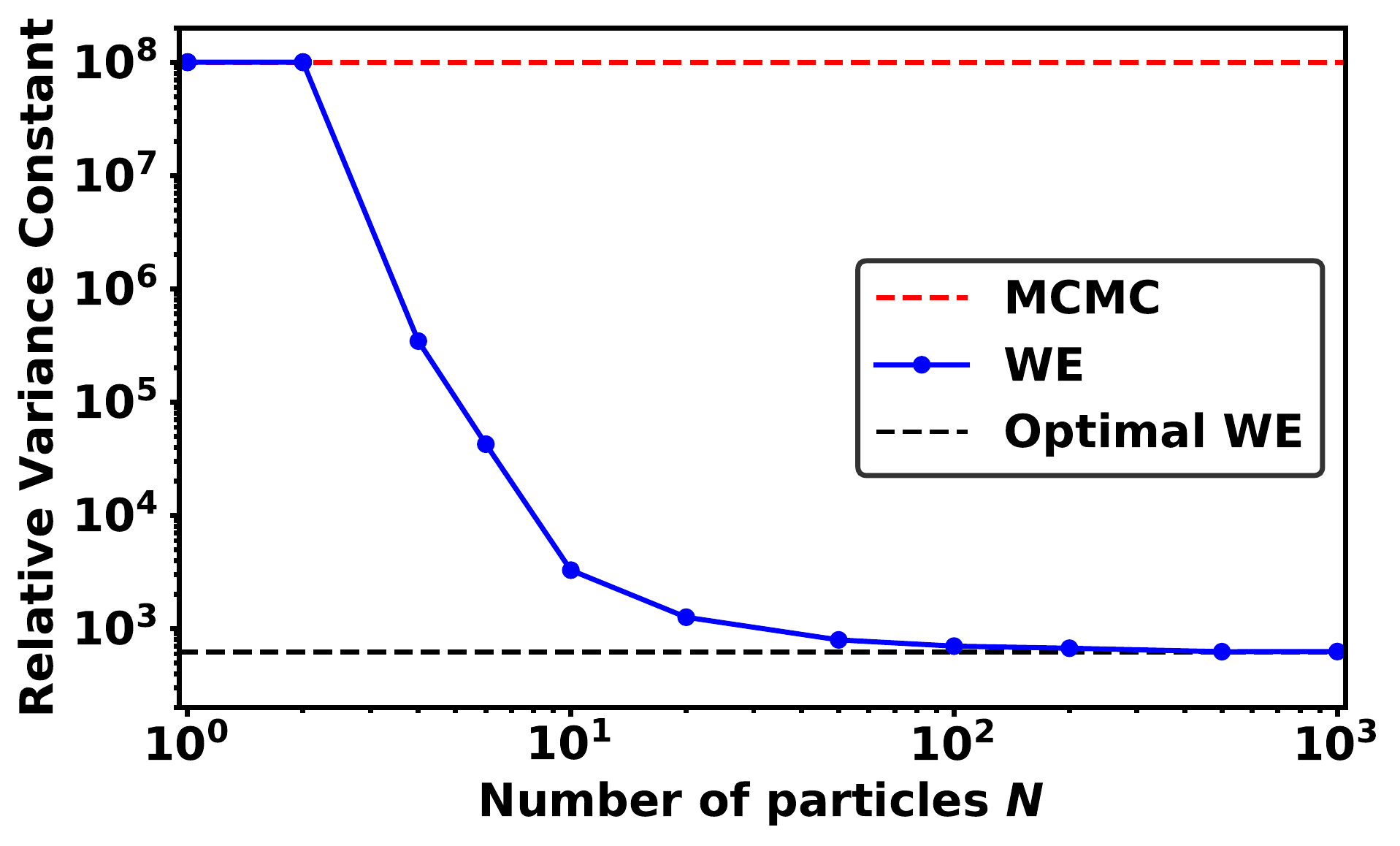}
	\caption{Application of WE to the geometric tails problem.}
	\label{fig:figure2}
\end{figure}

\subsection{Gaussian tail probabilities}\label{sec:OU}

In our second example, we use WE to estimate tail probabilities
$\mu\left[a, \infty\right)$ for the Gaussian distribution $\mu = \mathcal{N}\left(0, 1\right)$.
To sample from $\mu$, we use the first-order autoregressive process
\begin{equation}
\label{eq:ar1}
    {X}_{k+1} = e^{-\Dt} {X}_k + \sqrt{1 - e^{-2\Dt}} \eta_{k+1}, \quad \eta_{k+1}\sim N(0,1).
\end{equation}

\subsubsection{WE implementation}

In our numerical tests, we apply WE to estimate the tail probabilities
$\mu\left[3, \infty\right) = 1.35\times 10^{-3}$ and $\mu\left[4, \infty\right) = 3.17\times 10^{-5}$.
We start all the particles at $x = 0$,
and then we simulate forward for $n_T = T \slash \Dt$ time steps, where $T = 10^4$ and $\Dt = 0.01$.
At each splitting step,
we sort the particles into bins based on the intervals $\left(x_i, x_{i+1}\right]$, where
\begin{equation}
    -\infty = x_0 < x_1 < \cdots < x_{\max - 1} < x_{\max} = \infty.
\end{equation}
We optimize the mesh points $x_2, \ldots, x_{\max - 2}$ to ensure that intervals $\left(x_i, x_{i+1}\right]_{1 \leq i \leq \max - 2}$ are approximate level sets of $h_f$.
We use the {\tt WeightedEnsemble.jl} package \cite{WEjl} for our numerical implementation and describe additional implementation details in Appendix \ref{a:OU}.

\subsubsection{WE error bars}

In this example, 
we consider two data-driven strategies for estimating the variance of WE estimates.
As a first strategy, we run WE for $100$ independent trials
and apply a bootstrap approach for estimating the variance \cite{asmussen2007stochastic,davison1997bootstrap}.
In this bootstrap approach, we generate
$M = 10^4$ bootstrap samples of size $100$
by randomly subsampling from the independent WE estimates.
Then, for each bootstrap sample, we compute the empirical variance.
By aggregating together the $M = 10^4$ variance estimates,
we obtain a point estimate and robust confidence intervals for WE's variance.

As a second strategy for variance estimation,
we apply the following variance estimate to each one of the independent WE runs:
\begin{align}
    & \Var\left[\frac{1}{n_T} 
    \sum_{t=0}^{n_T-1} \sum_{i=1}^N w_t^i f\left(\xi_t^i\right)\right] \\
\label{e:IATTAVC}
    & \approx 
    \frac{1}{n_T^2} \sum_{\left|t - s\right| \leq L} \left(\sum_{i=1}^N w_t^i f\left(\xi_t^i\right) - \hat{\mu}\left(f\right)\right)
    \left(\sum_{i=1}^N w_s^i f\left(\xi_s^i\right) - \hat{\mu}\left(f\right)\right).
\end{align}
In this formula, 
\begin{equation}
\hat{\mu}\left(f\right) = \frac{1}{n_T} \sum_{t=0}^{n_T - 1} \sum_{i=1}^N w_t^i f\left(\xi_t^i\right)
\end{equation}
is the empirical estimate of $\mu\left(f\right)$, while
$L \geq 0$ is a truncation threshold,
chosen so that correlations between $\sum_{i=1}^N w_t^i f\left(\xi_t^i\right)$
and $\sum_{i=1}^N w_s^i f\left(\xi_s^i\right)$
are negligible for any time lag $\left|s - t\right|$ exceeding $L$.

The variance estimator \eqref{e:IATTAVC} 
is potentially very useful, since it provide error bars for WE estimates even after a single run of the algorithm.
Indeed, \eqref{e:IATTAVC} is already the standard variance estimator in MCMC,
and among MCMC practitioners it is known as the integrated autocorrelation time (IAT) estimator \cite{sokal1997monte}.
When the IAT estimator is applied to WE results, the full convergence properties have not yet been rigorously guaranteed.
However, we observe that the estimator has asymptotic bias that vanishes exponentially fast as we increase the truncation threshold $L$.
Moreover, in our experiments, we find good agreement between variance estimates using the IAT estimator and those obtained using the bootstrap. 
Our results provide empirical evidence that, at least for some problems to which the WE is applied, the IAT estimator is a useful tool.

\subsubsection{WE results}

In Figure \ref{fig:OUa3},
we present our estimates of the relative variance constant
\begin{equation}
\textup{Relative Variance Constant} = \frac{N T}{\mu\left(f\right)^2} \Var\left[\frac{1}{n_T} \sum_{t=0}^{n_T-1} \sum_{i=1}^N w_t^i f\left(\xi_t^i\right)\right],
\end{equation}
where $f\left(x\right) = \mathds{1}\left\{x \geq 3\right\}$ in the first scenario.
and $f\left(x\right) = \mathds{1}\left\{x \geq 4\right\}$ in the second scenario.
We compare our experimental estimates 
against asymptotic formulas for the relative variance
constant that are valid in the simultaneous limit as
$T \rightarrow \infty$ and $\Dt \rightarrow 0$.
A full derivation of these formulas appears in the appendix.

\begin{figure}[h!]
    \subfigure[]{\includegraphics[width=7cm]{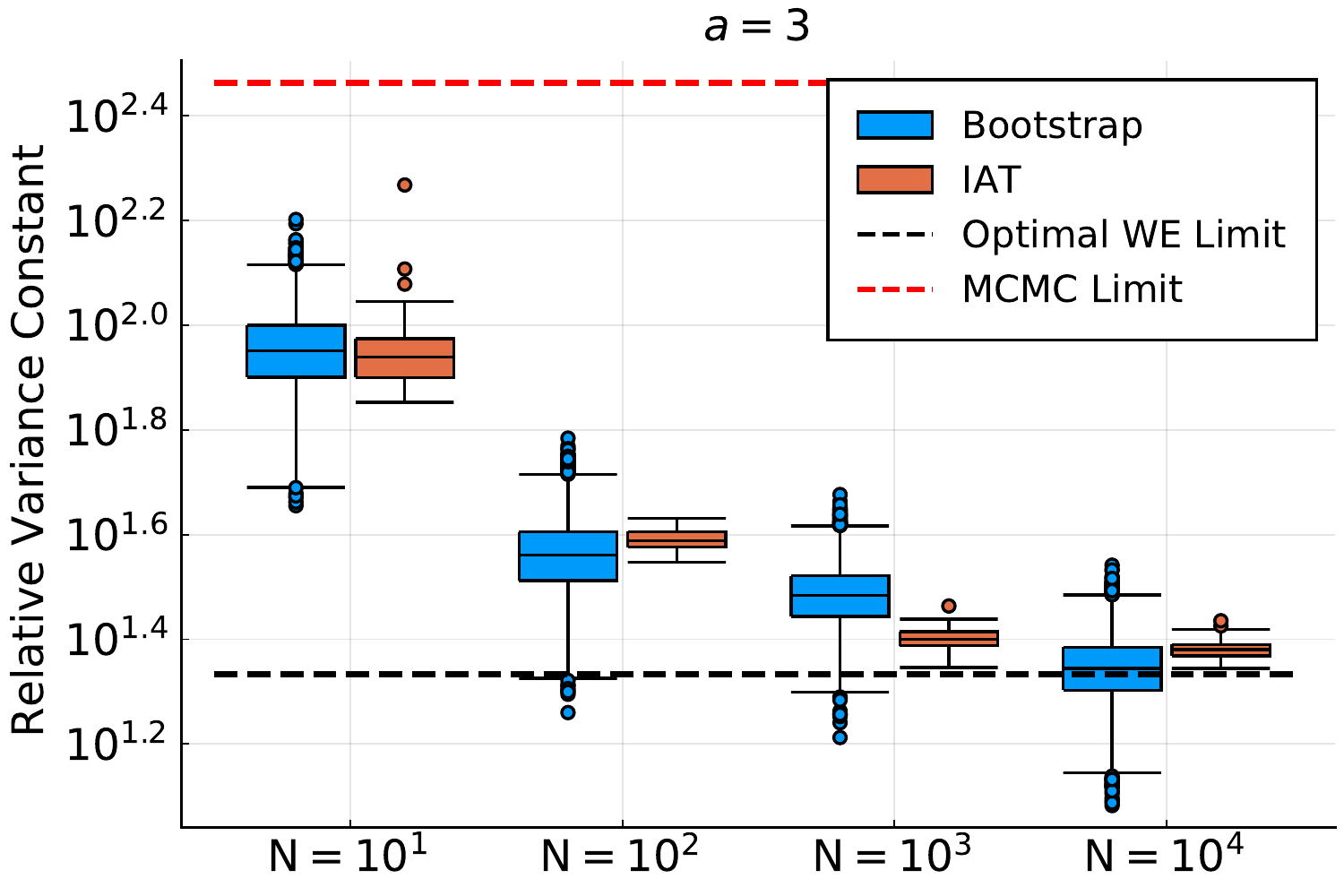}}
    \hfill
    \subfigure[]{\includegraphics[width=7cm]{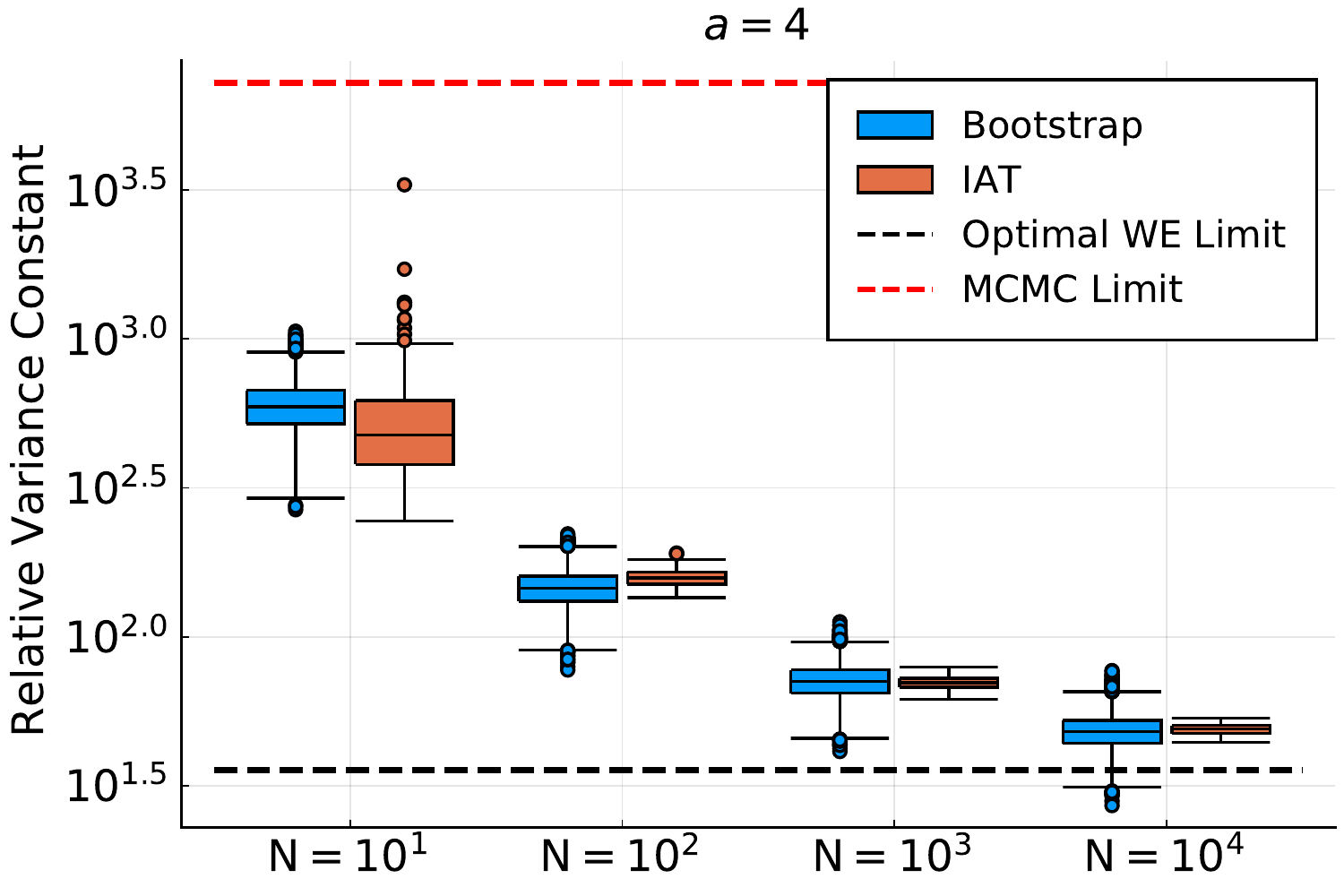}}
    \caption{Application of WE to the Gaussian tails problem.}
    \label{fig:OUa3}
\end{figure}

For $a = 3$ and sufficiently large $N$, WE nearly attains the optimal variance,
improving MCMC's variance by over an order of magnitude. 
For $a = 4$, WE's variance is somewhat further from the optimal variance,
yet WE still  achieves over two orders of magnitude improvement over direct MCMC sampling.

\subsection{Ising tail probabilities}{\label{sec:ising}}

In our third and final example, we use WE to calculate the probability of extreme magnetizations for the Ising model
on an $L \times L$ lattice with periodic boundary conditions. 
The Ising model has long been the subject of study in the statistical physics community as a model of ferromagnetism and as simple system exhibiting phase changes \cite{baxter2016exactly,gallavotti2013statistical,ruelle1999statistical}. 
The energy associated with the model is
\begin{equation}
\label{e:isingH}
    H(\bm{\sigma})= -\frac{1}{2} \sum_{i\sim j} \bm{\sigma}_i \bm{\sigma}_j, \quad \bm{\sigma}_i \in \left\{+1, -1\right\},
\end{equation}
where $i \sim j$ denotes that $i$ and $j$ are neighboring lattice points.
The associated Boltzmann distribution is 
\begin{equation}
    \mu\left(\bm{\sigma}\right) = \frac{\exp\left(-\beta H\left(\bm{\sigma}\right)\right)}
    {Z}, \quad Z = \sum_{\bm{\sigma}^{\prime}} \exp\left(-\beta H\left(\bm{\sigma}^{\prime}\right)\right).
\end{equation}
When $\beta> \beta_{\rm c}$
(the ``low-temperature" regime), the system tends to self-organize
with the majority of spins all either $+1$ or all $-1$.  
On the other hand, when $\beta < \beta_{\rm c}$
(the ``high-temperature" regime), self-organization is less likely, 
and a mixture of $+1$s and $-1$s becomes more likely.  

Our numerical tests address the following questions:
\begin{itemize}
    \item What is the probability that the mean magnetization, $m(\bm{\sigma}) = L^{-2}\sum_i \bm{\sigma}_i$, is in $(-0.1, 0.1)$ in the low-temperature regime?  In other words, what is the likelihood of seeing the system in a highly disordered state, despite being at low temperature?
    \item What is the probability that the mean magnetization satisfies $\left|m\right| > 0.9$ in the high-temperature regime?  Here, we are considering the likelihood of seeing the system in a highly ordered state, despite being at high temperature.
\end{itemize}

\subsubsection{WE implementation}
In our experiments, we implement WE on a $10 \times 10$ lattice.
In the low-temperature regime, we start $100$ particles from an initial state of all $-1$s.
In the high-temperature regime, we start $100$ particles 
from an initial state randomly selected from the uniform distribution on spins.
In both regimes, we evolve the particles forward 
by selecting one of the $L^2$ spins uniformly and 
proposing a flip from $\bm{\sigma}_i$ to $-\bm{\sigma}_i$.
We accept this proposed change with probability
\begin{equation}
    \min\left\{1, \exp\left(-\beta \bm{\sigma}_i \sum\nolimits_{j \sim i} \bm{\sigma}_j \right)\right\},
\end{equation}
and otherwise leave the system unchanged.
We perform ten such updates in each forward evolution step.
Then, in each splitting step,
we sort particles into bins based on mean magnetization and apply splitting and killing.
We describe additional details in Appendix \ref{a:Ising}.

\subsubsection{WE results}{\label{sec:ising_results}}

In Figures \ref{fig:lowT} and \ref{fig:highT},
we report the mean and standard deviation of the running averages
\begin{equation}
    \frac{1}{t} \sum_{s=0}^{t-1} \sum_{i=1}^N w_t^i f\left(\xi_t^i\right), \quad t = 0,1, \ldots T-1,
\end{equation}
computed over $100$ independent trials
for the functions $f\left(\bm{\sigma}\right) = \mathds{1}\left\{\left|m\left(\bm{\sigma}\right)\right| > 0.9\right\}$
and $f\left(\bm{\sigma}\right) = \mathds{1}\left\{\left|m\left(\bm{\sigma}\right)\right| < 0.1\right\}$.
We also report the relative variance constants based on the running averages.

\begin{figure}[h!]
    \subfigure[]{\includegraphics[width=7cm]{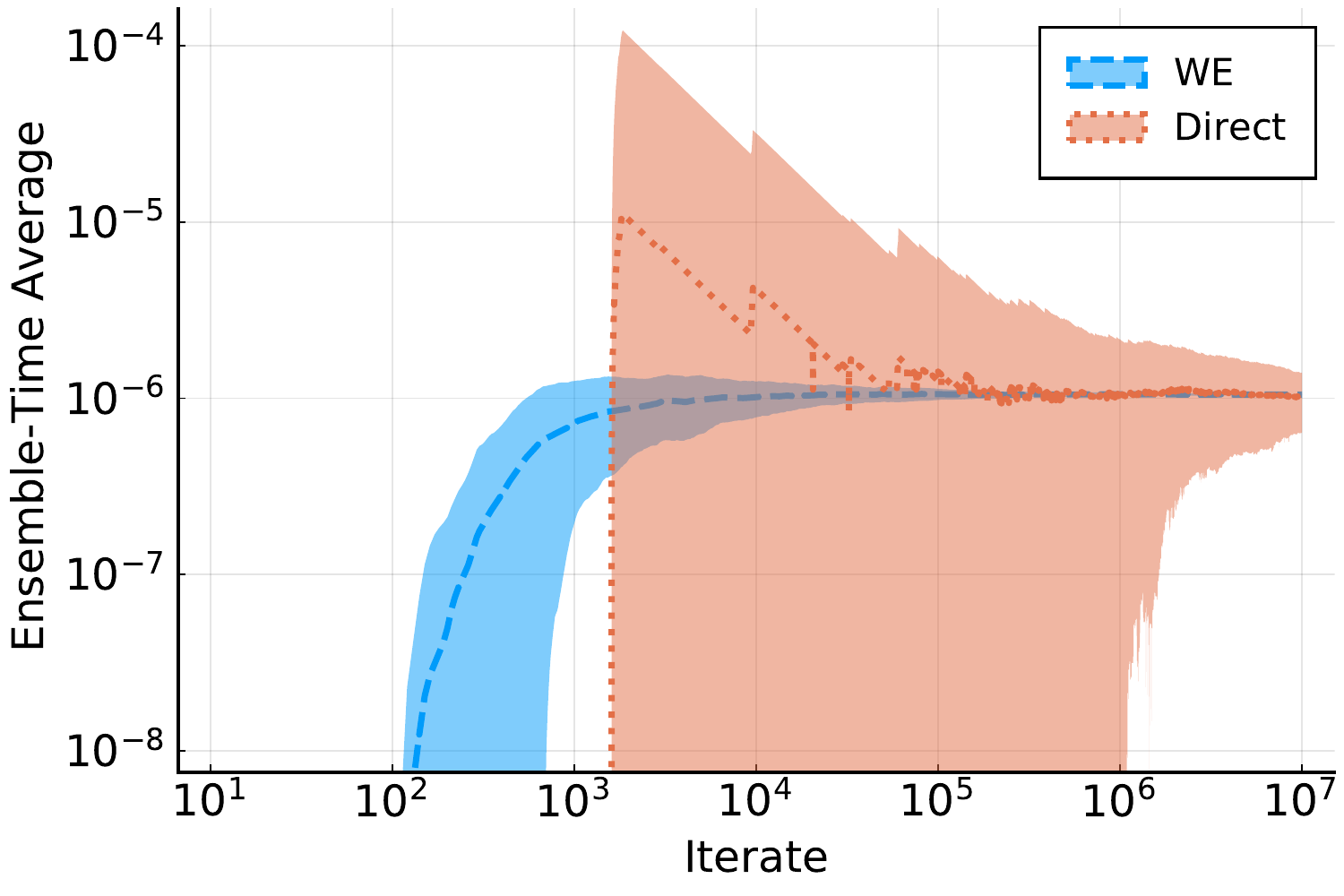}}
    \hfill
    \subfigure[]{\includegraphics[width=7cm]{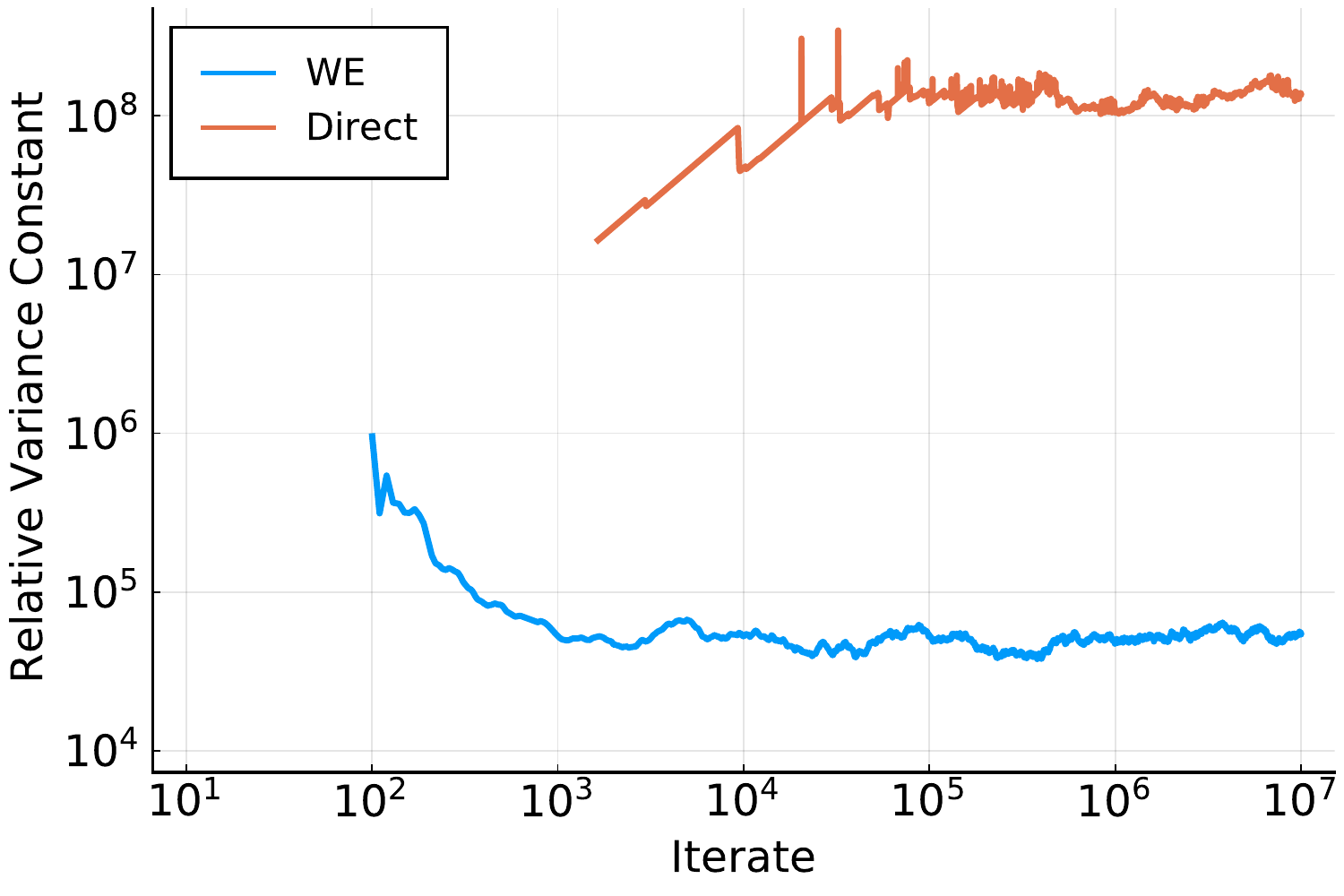}} 
    
    \caption{Application of WE to the Ising model at a low temperature ($\beta = 0.6$).}
    \label{fig:lowT}
\end{figure}

\begin{figure}[h!]
    \subfigure[]{\includegraphics[width=7cm]{hightemp1.pdf}}
    \hfill
    \subfigure[]{\includegraphics[width=7cm]{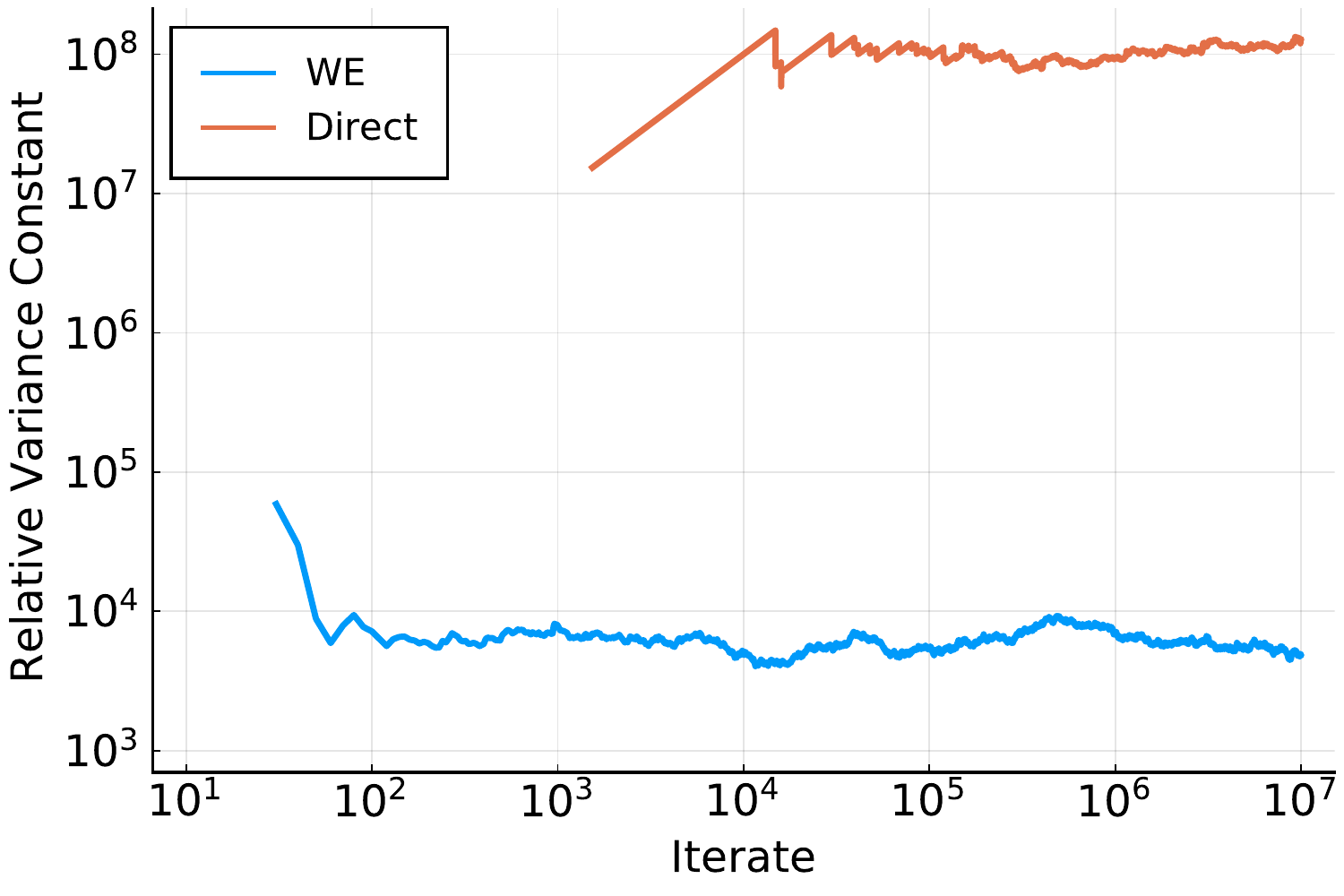}}
    
    \caption{Application of WE to the Ising model at a high temperature ($\beta = 0.25$).}
    \label{fig:highT}
\end{figure}

Not only do we find that WE is more computationally efficient than MCMC, but
our results also show that WE is more efficient than 
sampling from the Ising model by using an independence sampler.
An independence sampler would lead to a relative variance constant of $p^{-1} - 1$
when estimating a rare probability $p$.
Yet Figures \ref{fig:lowT} and \ref{fig:highT} show
that WE improves this variance constant by several orders of magnitude,
providing especially large improvements in
the high-temperature regime.
In conclusion, we obtain a remarkable result:
WE transforms the time correlations in the dynamics,
which would normally be an impediment to efficient sampling \cite{sokal1997monte},
into a major asset that enables significant variance reduction.

\section{Conclusion}{\label{sec:conclusion}}

In this work, we presented splitting as an approach for reducing MCMC’s variance when estimating rare event probabilities.
Traditionally, splitting is viewed as separate from MCMC in the mathematical literature.
However, here we showed that splitting can be beneficially combined with MCMC when appropriate stability conditions
are satisfied.
We contributed the following results:
\begin{enumerate}
    \item We showed that splitting schemes can degenerate over long timescales due to shrinking weights. Moreover, we proved that the only way to avoid shrinking weights is by using weighted ensemble (WE).
    \item We presented an optimal variance bound for WE that demonstrates the method’s maximal efficiency when a large number of particles are available.
    \item We explored numerical examples where WE reduces MCMC’s variance by multiple orders of magnitude. 
\end{enumerate}

As our numerical examples make clear,
there remain significant open questions for investigation.
First, it would be desirable to estimate the variance of WE estimates from a single long trajectory of WE data.
Yet it remains to be determined whether the integrated autocorrelation time (IAT) estimator 
provides convergent estimates of WE's variance.
Second, it is clear from our examples that WE requires a large number of particles in order to attain peak efficiency.
The precise scaling of the variance with the number of particles is an open area of investigation.

In light of these  open questions, we regard our present work not as the final answer regarding WE’s properties but rather as an essential step toward uncovering the method's mathematical foundations.
Here, we have demonstrated WE's importance as a practical computational tool and 
its interest as a mathematical system where
interactions perturb
the behavior of ergodic Markov chains.
We have shown that despite the apparent complexity of WE's dynamics, the mean and variance of WE's estimates
can be precisely bounded, yielding insights into the method’s efficiency.
In summary, we have established the unique role of WE as a splitting method that reduces MCMC variance
and constructed a rigorous framework that will aid in the method's future development.

\begin{appendix}
\section{Details of OU Computations}
\label{a:OU}

To calculate WE's optimal variance,
we use asymptotic approximations that are valid in
the simultaneous limit as $T \rightarrow \infty$ and $\Dt \rightarrow 0$.
We observe that the process \eqref{eq:ar1} is the $\Dt$-skeleton
of the continuous-time Ornstein-Uhlenbeck (OU) process
\begin{equation}
\label{e:OU}
  \mathop{d\bar{X}_t} = -\bar{X}_t \mathop{dt} + \sqrt{2}\mathop{d\bar{W}_t}.  
\end{equation}
Therefore, when $\Dt \ll 1$, we can approximate the conditional expectation function $h_f$ using
\begin{equation}
    \bar{h}_f = \frac{1}{\Dt} \E_x \left[ \int_0^{\infty} \mathds{1}\left\{\bar{X}_t \geq a\right\} - \mu\left[a, \infty\right) \mathop{dt} \right].
\end{equation}
Likewise, we can approximate the variance function $v_f^2 = K h_f^2 - \left(K h_f\right)^2$ using
\begin{equation}
\label{eq:quadratic_variation}
    \bar{v}_f^2\left(x\right) = \Dt \lim_{t \rightarrow 0+} \frac{1}{t} \E_x\left|\bar{h}_f\left(X_t\right) - \bar{h}_f\left(X_0\right)\right|^2.
\end{equation}
The approximation as $\Delta t \ll 0$ leads to useful simplifications,
since $\bar{v}_f^2$ is determined by the quadratic variation \cite{kallenberg2006foundations} of the process $\bar{h}_f\left(X_t\right)$; hence,
\begin{equation}
\label{eq:simplify}
    \bar{v}_f^2\left(x\right) = 2 \Delta t \left| \frac{\mathop{d \bar{h}_f\left(x\right)}}{\mathop{dx}} \right|^2
\end{equation}

To calculate the conditional expectation function $\overline{h}_f$ and the variance function $\overline{v}_f^2$ explicitly using Mathematica,
we first observe that
\begin{equation}
    \Delta t \overline{h}_f 
    = \int_0^{\infty} \left(P_x\left\{\overline{X}_t \geq a\right\} - \mu\left[a, \infty\right)\right) \mathop{dt}
\end{equation}
solves the Poison equation
\begin{equation}
-\mathcal{L} \left(\Delta t \overline{h}_f\right) = 
\mathds{1}\left\{x \geq a\right\} - \mu\left[a, \infty\right)
\end{equation}
involving the infinitesimal generator of the OU process
$\mathcal{L} g = -x g^{\prime} + g^{\prime \prime}$.
Hence, the approximate variance function $\overline{v}_f^2 = 2 \Dt \left|\overline{h}_f^{\prime}\right|^2$ solves the first-order ODE
\begin{equation}
x \overline{v}_f - \overline{v}_f^{\prime} = 
\sqrt{\frac{2}{\Dt}}\left(\mathds{1}\left\{x \geq a\right\} - \mu\left[a, \infty\right)\right).
\end{equation}
Solving the ODE gives \begin{equation}
\label{eq:soln_ode}
    \bar{v}_f\left(x\right) = 
    \sqrt{\frac{2}{\Dt}} \frac{ \min\left\{\Phi\left(x\right), \Phi\left(a\right)\right\} - \Phi\left(x\right) \Phi\left(a\right)}
    {\phi\left(x\right)},
\end{equation}
where
\begin{equation}
    \phi\left(x\right) = \frac{\exp\left(-x^2 \slash 2\right)}{\sqrt{2\pi}},
    \quad \text{and} \quad \Phi\left(x\right) = \int_{-\infty}^x \phi\left(y\right) \mathop{dy}
\end{equation}
are the probability density function and cumulative distribution function for a Gaussian distribution.
Using formula \eqref{eq:soln_ode}, 
we conclude that the MCMC variance and the optimal WE variance can be approximated as follows.
\begin{align}
    & \text{MCMC variance:} \quad &&
    \frac{\mu\left(\bar{v}_f^2\right)}{N T \slash \Dt} = \frac{4 \exp\left(-a^2\slash2\right)}{\sqrt{2 \pi} a^3 NT}\left(1 + \mathcal{O}\left(a^{-2}\right)\right). && \\
    & \text{Optimal WE variance:} \quad && \frac{\mu\left(\bar{v}_f\right)^2}{N T \slash \Dt} = \frac{\exp\left(-a^2\right)}{\pi NT}. &&
\end{align}
Thus, we find that the optimal improvement factor of WE over MCMC increases exponentially fast as $a \rightarrow \infty$.
Lastly, using Mathematica we integrate \eqref{eq:soln_ode} to obtain a closed-form expression for $\overline{h}_f$ involving confluent hypergeometric functions of the first kind.

In our implementation of WE, we define bins using a mesh
\begin{equation}
    -\infty = x_0 < x_1 < \cdots < x_{\max - 1} < x_{\max}. = \infty.
\end{equation}
The endpoints of the mesh are set to $x_1 = -2$ and $x_{\max} = 3.5$ in the case $a = 3$, and $x_1 = -2$ and $x_{\max} = 5$ in the case $a = 4$.
The interior mesh points $x_2, x_3, \ldots, x_{\max - 1}$ are chosen to constrain the variation of $\Dt \overline{h}_f$ over each of the intervals $\left(x_i, x_{i+1}\right]_{1 \leq i \leq \max - 2}$.
The variation per interval is set to $10^{-3}$ in the case $a = 3$
and $10^{-4}$ in the case $a = 4$.

Lastly, during the WE run, we allocate children particles to each bin according to the rule
\begin{equation}
    \frac{N_t\left(u\right)}{N} \approx \frac{w_t\left(u\right) \eta_t^u\left(\overline{v}_f\right)}
    {\sum_{u^{\prime}} w_t\left(u^{\prime}\right) \eta_t^{u^{\prime}}\left(\overline{v}_f\right)},
\end{equation}
as described in \cite{aristoff2020optimizing}.
We use systematic resampling to select particles within the bins.

\section{Details of Ising Computations}
\label{a:Ising}

We set the bins to be Voronoi cells 
in the magnetization coordinate $m$
with centers $-1, -0.9, \ldots, 0.9, 1$.
We allocate children particles to each bin according to the rule
\begin{equation}
    \frac{N_t\left(u\right)}{N} \approx \frac{w_t\left(u\right) \eta_t^u\left(\overline{v}_f\right)}
    {\sum_{u^{\prime}} w_t\left(u^{\prime}\right) \eta_t^{u^{\prime}}\left(\overline{v}_f\right)},
\end{equation}
where $\overline{v}_f$ is an approximation to $v_f$ built on a coarse model of the dynamics.

To obtain $\overline{v}_f$, we follow the microbin approach developed in \cite{aristoff2018analysis,aristoff2020optimizing}.
We first use short, independent simulations to obtain a transition matrix $\overline{K}$
for the coordinate $m$.
Specifically, by sampling from the uniform distribution with fixed magnetization $m$,
we obtain $10^4$ initial data points in each magnetization state
\begin{equation}
\label{e:microbins1}
    m = -1, -1 + 2L^{-2},\ldots 1 - 2L^{-2}, 1.
\end{equation}
Then, we run the dynamics forward for one evolution step to estimate the entries
\begin{equation}
    \overline{K}_{ij} = \sum_{m\left(\bm{\sigma}\right) = i} \mu\left(\bm{\sigma}\right) K\left(\bm{\sigma}, \mathds{1}\left\{m = m_j\right\}\right).
\end{equation}
We show the estimated $\overline{K}$ matrix in Figure \ref{fig:bins} below.

\begin{figure}[h!]
    \includegraphics[width=7cm]{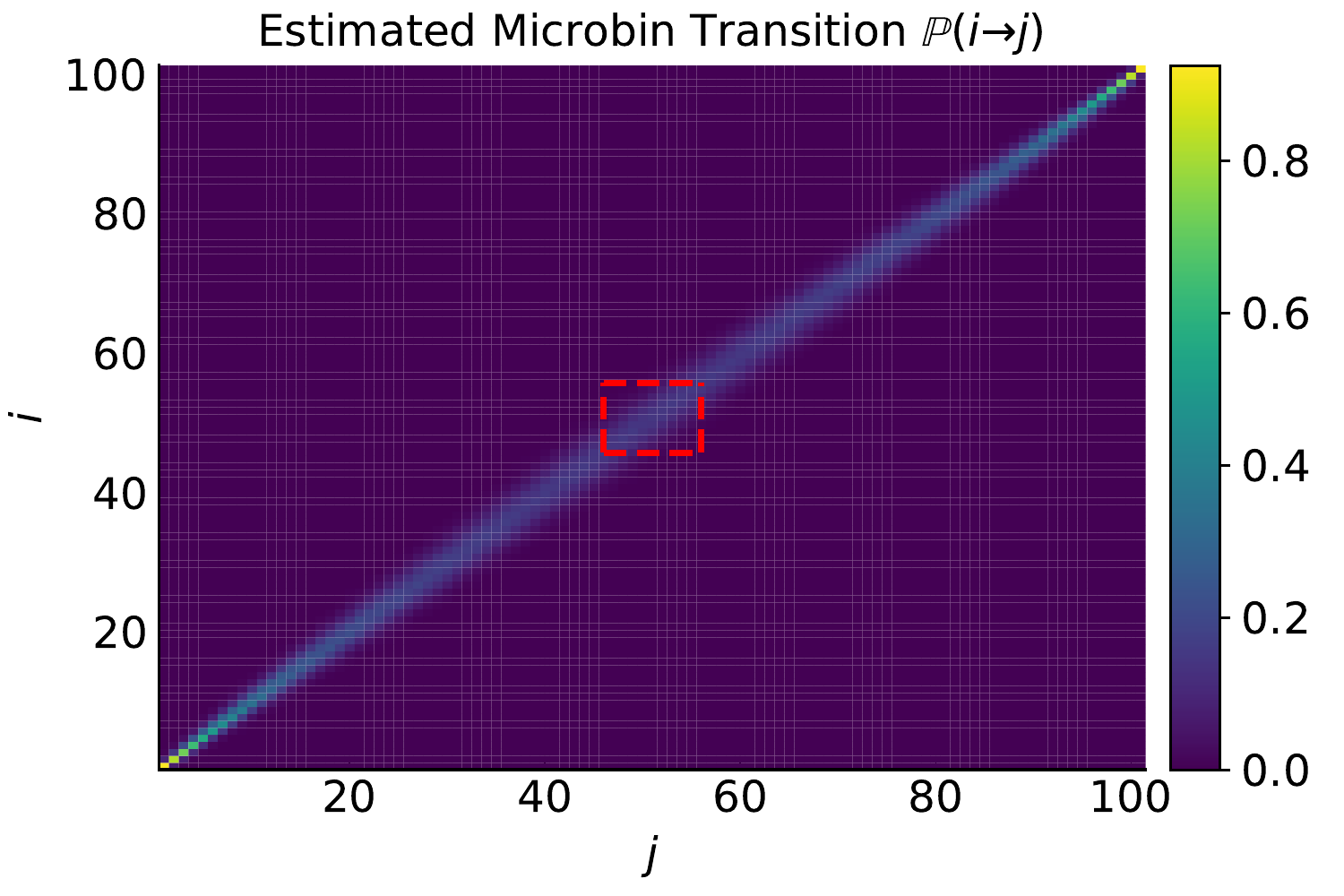}
    \caption{Microbin transition matrix for the low-temperature Ising model.
    The red square indicates the low-magnetization state $\left|m\right| < 0.1$ whose probability we seek to estimate.}
    \label{fig:bins}
\end{figure}

Having obtained $\overline{K}$, the microbin transition matrix, we next compute the microbin invariant measure $\overline{\mu}^T = \overline{\mu}^T \overline{K}$.
Lastly, we solve the Poisson equation
\begin{equation}
\left(I - \overline{K}\right) \overline{h}_f = f - \overline{u}\left(f\right)
\end{equation}
to approximate the conditional expectation function $\overline{h}_f$ and the variance function 
$\overline{v}_f\left(x\right)^2 = \Var_{\overline{K}\left(x, \cdot\right)}\left[\overline{h}_f\right]$.

\end{appendix}

%
%

\section*{Acknowledgements}
The authors would like to thank Aaron Dinner and Jonathan Weare
for helpful conversations.
RJW was supported by the National Science Foundation award DMS-1646339
and by New York University's Dean's Dissertation Fellowship.
DA and GS were supported by the National Science Foundation award DMS-1818726.
 


\bibliographystyle{imsart-number} 
\bibliography{references}       


\end{document}